\documentclass[12pt]{article}

\usepackage{amsmath,amsthm,amssymb}
\allowdisplaybreaks
\usepackage{mathrsfs}
\usepackage{cite}

\usepackage[pdftex,bookmarksopen=true,bookmarks=true,unicode,setpagesize]{hyperref}
\hypersetup{colorlinks=true,linkcolor=black,citecolor=black}

\newtheorem{theorem}{Theorem}
\newtheorem{corollary}[theorem]{Corollary}
\newtheorem{lemma}[theorem]{Lemma}
\newtheorem{proposition}[theorem]{Proposition}

\theoremstyle{remark}

\theoremstyle{remark}
\newtheorem{example}[theorem]{Example}
\theoremstyle{remark}
\newtheorem{remark}[theorem]{Remark}

\newcommand{\vol}{\operatorname{vol}}

\newcommand{\R}{\mathbb R}
\newcommand{\Dif}{\operatorname{Diff}}

\begin{document}

\vspace{-20mm}
\begin{center}{\Large \bf
Quasi-invariance of completely random measures}
	\end{center}

{\large Habeebat O. Ibraheem}\\ Department of Mathematics,
Swansea University, Singleton Park, Swansea SA2 8PP, U.K.;
e-mail: \texttt{e.lytvynov@swansea.ac.uk}\vspace{2mm}

{\large Eugene Lytvynov}\\ Department of Mathematics,
Swansea University, Singleton Park, Swansea SA2 8PP, U.K.;
e-mail: \texttt{e.lytvynov@swansea.ac.uk}\vspace{2mm}

{\small

\begin{center}
{\bf Abstract}
\end{center}
\noindent 
Let $X$ be a locally compact Polish space. Let $\mathbb K(X)$ denote the space of discrete Radon measures on $X$. Let $\mu$ be a completely random discrete measure on $X$, i.e., $\mu$ is (the distribution of) a completely random measure on $X$ that is concentrated on $\mathbb K(X)$. We consider the multiplicative (current) group $C_0(X\to\R_+)$ consisting of functions on $X$  that take values in $\R_+=(0,\infty)$ and are equal to 1 outside a compact set. 
 Each element $\theta\in C_0(X\to\R_+)$ maps $\mathbb K(X)$ onto itself; more precisely, $\theta$ sends a discrete Radon measure $\sum_i s_i\delta_{x_i}$ to $\sum_i \theta(s_i)s_i\delta_{x_i}$. 
 Thus, elements of $C_0(X\to\R_+)$ transform the weights of  discrete Radon measures.
 We study conditions under which the measure $\mu$ is quasi-invariant under the action of the current group $C_0(X\to\R_+)$ and consider several classes of examples. We further assume that $X=\R^d$ and consider the group of local diffeomorphisms $\Dif_0(X)$. Elements of this group also map  $\mathbb K(X)$ onto itself. More precisely, a diffeomorphism $\varphi\in \Dif_0(X)$ sends a discrete Radon measure $\sum_i s_i\delta_{x_i}$ to $\sum_i s_i\delta_{\varphi(x_i)}$. Thus, diffeomorphisms from $\Dif_0(X)$ transform the atoms of  discrete Radon measures. We study quasi-invariance of $\mu$ under the action of $\Dif_0(X)$. We finally consider the semidirect product $\mathfrak G:=\Dif_0(X)\times C_0(X\to \mathbb R_+)$ and study conditions of quasi-invariance and partial quasi-invariance of $\mu$ under the action of $\mathfrak G$.
	 } \vspace{2mm}
	


	\section{Introduction}
	
Let $P$ be a probability measure on a sample space $\Omega$ and let $\mathfrak G$ be a group acting on $\Omega$. A fundamental question of the representation theory is whether the probability measure $P$ is quasi-invariant with respect to this action. The latter means that, for each element $g\in\mathfrak G$, the pushforward of $P$ under $g$, denoted by $P^g$, is  equivalent to the measure $P$, so that the Radon--Nikodym density $\frac{dP^g}{dP}$ exists and is strictly positive $P$-a.e. If this holds, one can construct a unitary representation of the group $\mathfrak G$ in $L^2(\Omega,P)$. To this end, for each $g\in\mathfrak G$, one defines a unitary operator $U_g$ in $L^2(\Omega,P)$ by 
$$(U_gf)(\omega)=f(g^{-1}\omega)\sqrt{\frac{dP^g}{dP}(\omega)}.$$
Such a representation of $\mathfrak G$ is sometimes called quasi-regular.

In the case where the group $\mathfrak G$ is big, the problem of quasi-invariance of $P$ with respect to the action of $\mathfrak G$ may be very difficult.

Let us consider an important example of such a construction. Let $X=\mathbb R^d$  and let $dx$ be the Lebesgue measure on $X$. Denote by $\Omega=\Gamma(X)$ the space of locally finite subsets of $X$ (configurations). Let $P=\pi_z$ be the Poisson measure on $X$ with intensity measure $z\,dx$, where $z>0$ is a fixed constant.  
Let $\mathfrak G=\Dif_0(X)$ be the group of diffeomorphisms of $X$ which are equal to the identity outside a compact set. Elements of $\varphi\in\Dif_0(X)$ naturally act on $\Gamma(X)$ by moving each point of the configuration. The measure $\pi_z$ appears to be quasi-invariant with respect to the action of $\Dif_0(X)$. In particular, for each $\varphi\in \Dif_0(X)$, the Radon--Nikodym derivative is given by
$$\frac{d\pi_z^\varphi}{d\pi_z}(\gamma)=\prod_{x\in\gamma}J_\varphi(x),\quad\gamma\in\Gamma(X).$$
  Here $J_\varphi$ is the modulus of the determinant of the Jacobian matrix of $\varphi$. 
 As a result, we construct a unitary representation of $\Dif_0(X)$ in $L^2(\Gamma(X),\pi_z)$.

The problem of representations of the group of diffeomorphisms of a smooth (noncompact) Riemannian manifold $X$ in the $L^2$-space with respect to a Poisson measure is a classical one. The fundamental paper   \cite{VGG} by Vershik, Gel'fand, and Graev is a standard reference here.

Let us note that  representations of the semidirect product of the additive group $C^\infty(X)$  and $\Dif_0(X)$ in $L^2(\Gamma(X),P)$ are important for nonrelativistic quantum mechanics, see e.g.\ \cite{G1,GGPS,GMS} and the references therein. Here $P$ is a probability measure on the configuration space $\Gamma(X)$, in particular, $P$ can be a Poisson measure.  

The representations of the diffeomorphism group $\Dif_0(X)$ in the $L^2$-space with respect to a Poisson measure naturally led Albeverio, Kondratiev and R\"ockner \cite{AKR1,AKR2} to defining elements of differential geometry on the configuration space $\Gamma(X)$ (directional derivative, gradient, tangent space), and developing related analysis on the configuration space equipped with Poisson measure, or more generally, with a Gibbs measure (the Laplace operator, the heat semigroup), and studying the corresponding stochastic processes (Brownian motions) on the configuration space, see also  \cite{KLR,MR,AKR3,free_dynamics}. Laplace operators on the differential forms over the configuration space $\Gamma(X)$ equipped with Poisson measure (and more generally, with a Gibbs measure) were studied by Albeverio, Daletskii and Lytvynov in \cite{ADL1,ADL2,ADL3}.

Tsilevich, Vershik, and Yor \cite{bibb3} studied quasi-invariance of the gamma measure with respect to the action of the multiplicative group $C_0(X\to\R_+)$. This group consists of functions on $X$ which take values in $\R_+$ and are equal to 1 outside a compact set. The gamma measure is a random measure on $X$; it belongs to the class of measure-valued L\'evy processes. This random measure takes almost surely values in the space $\mathbb K(X)$ of discrete Radon measures on $X$. The latter space consists of Radon measures of the form $\sum_i s_i\delta_{x_i}$, where $s_i>0$ and $\delta_{x_i}$ is the Dirac measure with mass at $x_i$. Each element $\theta\in C_0(X\to\R_+)$ maps $\mathbb K(X)$ onto itself; more precisely, $\theta$ sends the discrete Radon measure $\sum_i s_i\delta_{x_i}$ to $\sum_i \theta(s_i)s_i\delta_{x_i}$. The (distribution of) the gamma measure appears to be quasi-invariant under the action of $C_0(X\to\R_+)$.

One can naturally define the semidirect product $\mathfrak G$ of the diffeomorphism group $\Dif_0(X)$ and $C_0(X\to\R_+)$. This group consists of all pairs $(\varphi,\theta)\in \Dif_0(X)\times C_0(X\to\R_+)$ and it naturally acts on the space of discrete Radon measures, $\mathbb K(X)$: for each $\sum_i s_i\delta_{x_i}\in\mathbb K(X)$, its image under the action of $(\varphi,\theta)$ is equal to 
$\sum_i \theta(\varphi(x_i))s_i\delta_{\varphi(x_i)}$. However, it appears that, if the underlying space $X$ is not compact, the gamma measure is not quasi-invariant with respect to the action of $\mathfrak G$. Kondratiev, Lytvynov, and Vershik \cite{ba1} suggested the notion of partial quasi-invariance and proved that the gamma measure, and more generally, a class of measure-valued L\'evy processes, are partially quasi-invariant with respect to the action of $\mathfrak G$. 
The main point of this definition is that, despite absence of quasi-invariant, one can still derive analysis and geometry on space $\mathbb K(X)$ equipped with such a measure. One can again construct a gradient, a tangent space, and an associated Laplace operator on $\mathbb K(X)$, see  \cite{ba1}. Markov processes on $\mathbb K(X)$ which correspond to these Laplace operators are constructed by Conache,  Kondratiev, and Lytvynov \cite{CKL}.

Measure-valued L\'evy processes form a subclass of completely random measures. A completely random measure \cite{Kal,bib0,bib01} is a random measure on $X$ whose values are independent on mutually disjoint sets. We will actually deal with the important class of completely random measures which are discrete Radon measures, i.e., their distribution, $\mu$, is a probability measure on $\mathbb K(X)$.

The main problem we solve in this paper is: Under which conditions is a completely random discrete measure $\mu$ quasi-invariant, or partially quasi-invariant with respect to the action of $\mathfrak G$, the semidirect product of the diffeomorphism group $\Dif_0(X)$ and $C_0(X\to\R_+)$? Our results here extend the related results of \cite{ba1}. We also refer  to the papers \cite{KLU,AKLU} which discuss quasi-invariance of a compound Poisson process with respect to the action of the group $\mathfrak G$, or its generalization where $\mathbb R_+$ is replaced with   a Lie group. Also the results on quasi-invariance of the gamma measure with respect to the action of $C_0(X\to\R_+)$ were extended to Poisson processes on $X\times\R_+$ by    Lifshifts and  Shmileva \cite{bib1}. (Note that the problem
of quasi-invariance of a completely random discrete measure is related to the problem of quasi-invariance of the Poisson process on $X\times\R_+$.)

The paper is organized as follows. In Section \ref{yjd7k}, we recall the 
main notions related to completely random measures.  We fix a locally compact Polish space $X$ with its Borel $\sigma$-algebra $\mathcal B(X)$.  We denote by $\mathbb M(X)$ the set of all Radon measures on $X$, and the Borel $\sigma$-algebra on $\mathbb M(X)$ is denoted by $\mathcal B(\mathbb M(X))$. We define a random  measure as a measurable mapping from a probability space that takes values in $\mathbb M(X)$. Since we are only interested in the distribution of such a mapping, we agree to call any probability measure $\mu$ on $(\mathbb M(X), \mathcal B(\mathbb M(X)))$ a random measure. 
We define the configuration space $\Gamma(X)$ as a subset of  $\mathbb M(X)$, and we define a (simple) point process as a random measure which is concentrated on $\Gamma(X)$. We further recall the notion of  the  Poisson point process $\pi_\sigma$ with intensity measure $\sigma$. Here $\sigma$ is a non-atomic Radon measure on $(X,\mathcal B(X))$. We discuss the classical result about equivalence of two Poisson point processes,  $\pi _\rho$ and $\pi_\lambda$, \cite{bib2,bib3,bib1}. We also discuss the notion and construction  of a completely random measure \cite{bib0}. 
  We finally define a completely random discrete measure as a completely random measure which is concentrated on $\mathbb K(X)$.

 The results of the paper are in Sections~\ref{cte65i}--\ref{ydr637}. Here we study quasi-invariance of completely random discrete measures.

 In Section \ref{cte65i}, we assume that $X$ is a locally compact Polish space, and   we present sufficient conditions for a completely random discrete measure to be quasi-invariant under the action of the group $C_0(X\to\mathbb R_+)$ onto $\mathbb K(X)$ (transformations of weights of  Radon measures). 
 
 Note that, for  measure-valued L\'evy  processes, several conditions of their quasi-invariance  under the action of the group $C_0(X\to\mathbb R_+)$ onto $\mathbb K(X)$ were derived in \cite{ba1}. For a measure-valued L\'evy process, its L\'evy measure is  a measure on $X\times\R_+$ which is a product measure:  $dm(x,s)=d\sigma(x)\,d\nu(s)$, where $\sigma$ is a reference measure on $X$, while the L\'evy process is determined by the measure $\nu$ on $\R_+$. However, for a general completely random measure, its L\'evy measure $m$ does not have anymore the product structure. This creates  technical difficulties when discussing their quasi-invariance. So  in Section \ref{cte65i} we overcome these problems and present a number of criteria of the quasi-invariance of general completely random measures.

 We also consider three classes of  examples of application of these results. We first  discuss quasi-invariance of a completely random gamma measure. The latter random measure has the property that its L\'evy measure is a measure on $X\times\mathbb R_+$ of the form
 $$dm(x,s)=\beta(x)\,\frac{e^{-s/\alpha(x)}}s\,d\sigma(x)\,ds,$$
 where $\sigma$ is a fixed nonatomic Radon measure on $X$ (typically $d\sigma(x)=dx$ if $X=\mathbb R^d$) and $\alpha:X\to\mathbb R_+$ and $\beta:X\to[0,\infty)$ are measurable functions satisfying certain conditions.

Next, we consider a  class of completely random measures whose L\'evy measure is such that, for small values of $s$,
 $$dm(x,s)=\beta(x)(-\log s)^{\alpha(x)}\,d\sigma(x)\,ds,$$
 where $\alpha,\beta:X\to\mathbb R_+$ are measurable functions satisfying certain conditions. 
 
 And finally, we consider a  class of completely random measures whose L\'evy measure is such that, for small values of $s$,
 $$dm(x,s)=\beta(x)s^{1-\alpha(x)}\,d\sigma(x)\,ds,$$
 where $\alpha:X\to(0,1)$ and $\beta:X\to\mathbb R_+$ are measurable functions satisfying certain conditions.
 
    In Section \ref{ty77o897}, we assume that $X=\mathbb R^d$  and   we present sufficient conditions for a completely random discrete measure to be quasi-invariant under the action of  the diffeomorphism group $\Dif_0(X)$ 
 onto $\mathbb K(X)$ (transformations of atoms of  Radon measures).  We also consider applications of these results to the the three classes of examples we mentioned above.

 Finally, in Section \ref{ydr637}, we discuss quasi-invariance and partial quasi-inva\-riance 
 of a completely random discrete measure  under the  action of the semidirect product  $\mathfrak G$  of the groups $C_0(X\to\mathbb R_+)$ and $\Dif_0(X)$ onto $\mathbb K(X)$ (transformations of both weights and atoms of  Radon measures), and we also consider examples.

\section{Completely random measures}\label{yjd7k}

Let $X$ be a locally compact Polish space, and let $\mathcal B(X)$ denote the Borel $\sigma$-algebra on $X$.  
A measure $\eta$ on $(X,\mathcal{B}(X))$ is called a {\it Radon measure} if $\eta (\Lambda) < \infty$ for any compact  $\Lambda \subset X$. We denote by $\mathbb{M}(X)$ the set of all Radon measures on $X$. One defines the {\it vague topology on $\mathbb{M}(X)$} as the weakest topology on $\mathbb{M}(X)$ with respect to which any mapping of the following form is continuous:
\begin{equation}
\label{vg11}
\mathbb{M}(X)\ni \eta \mapsto \int _{X}f\,d\eta =:\langle f,\eta\rangle \in \mathbb{R}. 
\end{equation}
Here $f \in C_0(X)$, i.e., $f$ is a continuous function $f:X\rightarrow \mathbb{R}$ with compact support. We denote by $\mathcal{B}(\mathbb{M}(X))$ the Borel $\sigma$-algebra on $\mathbb{M}(X)$.

\begin{remark}
There is another way of characterization of $\mathcal{B}(\mathbb{M}(X))$. We denote by $\mathcal{B}_0(X)$ the collection of all sets from $\mathcal{B}(X)$ which have compact closure. Then one can show (see e.g.\ \cite{Kal}) that $\mathcal{B}(\mathbb{M}(X))$ is the minimal $\sigma$-algebra on $\mathbb{M}(X)$ with respect to which every mapping of the following form is measurable:
$$\mathbb{M}(X)\ni \eta\mapsto \eta(\Lambda)={\langle}\chi_{\Lambda}, \eta {\rangle} \in \mathbb{R},$$
 for each $\Lambda \in \mathcal{B}_0(X)$. Here $\chi _\Lambda$ denotes the indicator function of $\Lambda$.
 \end{remark}

Let $(\Omega,\mathcal{F},P)$ be a probability space. A measurable mapping 
$\xi :\Omega\to\mathbb{M}(X)$ is called a {\it random measure}. 
In most cases, we will only be interested in the distribution of a random measure on $\mathbb M(X)$. This is why we will often think of a random measure as a probability measure $\mu$ on $(\mathbb M(X),\mathcal B(\mathbb M(X)))$. In the latter case, $(\Omega,\mathcal{F},P)=(\mathbb M(X),\mathcal B(\mathbb M(X)),\mu)$ and the mapping $\xi$ is just the identity.

 Next, we will discuss a special subset of the set of random measures known as (simple) point processes. 
 The {\it configuration space over $X$} is defined by
 $$\Gamma(X):=\{\gamma\subset X\mid  |\gamma\cap\Lambda|<\infty\text{ for each compact } \Lambda\subset X\}.$$
 Here, for a set $A$, $|A|$ denotes the cardinality of $A$.
  Elements $\gamma$ of $\Gamma (X)$ are called {\it configurations in $X$}. 
  One identifies a configuration $\gamma \in \Gamma (X)$ with the  measure $\sum _{x \in \gamma}\delta_ {x}$. Here $\delta _x $ is the Dirac measure with mass at $x$.
Since a configuration $\gamma$ contains a finite number of points in each compact set, the measure $\sum _{x \in \gamma}\delta_ {x}$ is Radon. Hence, in the sense of this identification, we get the inclusion $\Gamma (X)\subset \mathbb{M}(X)$.

On $\Gamma (X)$ one defines the {\it vague topology} as  the trace of the vague topology on $\mathbb{M}(X)$. That is, the vague topology on $\Gamma (X)$ is the weakest topology on $\Gamma (X)$ with respect to which every mapping of the following form is continuous:
 $$\Gamma(X)\ni\gamma\mapsto\langle f,\gamma\rangle=\sum_{x \in  \gamma}f({x})\in \mathbb{R},$$
 where $f\in C_0(X)$. 
 One denotes by $\mathcal{B}(\Gamma (X))$ the corresponding Borel $\sigma$-algebra on $\Gamma (X)$. One can show that $\Gamma (X) \in \mathcal{B}(\mathbb{M}(X))$ and $\mathcal{B}(\Gamma(X))$ is the trace $\sigma$-algebra of $\mathcal{B}(\mathbb{M}(X))$ on $\Gamma (X)$.

 Let $(\Omega,\mathcal{F},P)$ be a probability space. A measurable mapping 
$\gamma :\Omega\to\Gamma(X)$ is called a {\it (simple) point process}. In particular, a point process is a random measure.
 Similarly to the case of random measures, we will often understand by a point process a  
 probability measure $\mu$  on $(\Gamma (X),\mathcal{B}(\Gamma(X)))$.

Let $\sigma$ be a Radon measure on $(X,\mathcal{B}(X))$ and let us assume that $\sigma$ is nonatomic, i.e., $\sigma(\{x\})=0$ for every $x\in X$. A {\it Poisson point process with intensity measure $\sigma$} is defined as the unique probability measure $\pi_{\sigma}$ on $\Gamma(X)$ which has Fourier transform
\begin{equation}
\int_{\Gamma(X)}e^{i\langle f,\gamma \rangle}\,d\pi_{\sigma}(\gamma)=\exp\left[\int_X(e^{if(x)}-1)d\sigma (x)\right] \label{a1}
\end{equation} for all $f\in C_0(X)$. See e.g.\cite{bib01} for further details.

 Let $\rho$ and $\lambda$ be  non-atomic Radon measures on $X$. Then we can construct  Poisson point processes (or Poisson measures) on $\Gamma(X)$ with intensity $\rho$ and $\lambda$, respectively, denoted by $\pi _\rho$ and $\pi _\lambda$. Now, the following question arises:  When are these measures equivalent, i.e. when is $\pi_\rho$ equivalent to $\pi_\lambda$? The  theorem below follows from Skorohod's result \cite{bib2}, from its extension by Takahashi \cite{bib3} to the case of a rather general underlying space, and from  Lifshits and Shmileva's result \cite[Theorem 2]{bib1}.

 \begin{theorem}\label{vufr76}
 Let $X$ be a locally compact  Polish space. Let $\rho$ and $\lambda$ be non-atomic Radon measures on $(X,\mathcal B(X))$.
The Poisson  measures $\pi_\rho$ and $\pi_\lambda$ are equivalent if and only if:
\begin{enumerate}
\item $\rho$ and $\lambda$ are equivalent;
\item if density $\phi:=\frac{d\rho}{d\lambda}$, then
\begin{equation}\label{poi2}
\int_X\left(\sqrt{\phi}-1\right)^2d\lambda<\infty.\end{equation}
\end{enumerate}
In the latter case,
\begin{align}
\frac{d\pi_\rho}{d\pi_\lambda}&=\exp \bigg[(\lambda - \rho)(X\setminus A)+\int_{X\setminus A}\log\phi\, d\lambda \notag\\
&\quad  +\int _ A(\log \phi -\phi +1)d\lambda +\int _ A \log\phi\, d(\gamma - \lambda)\bigg],\label{f7ur76}
\end{align}
 where 
 $A:=\left\{x\in X \mid |1-\phi(x)|< \frac{1}{2}\right\}$.
 \end{theorem}

\begin{remark}\label{prop1}
As easily seen, if we assume that
\begin{equation}\label{poi3}
 \phi -1\in L^1(X,\lambda), \end{equation}
 then condition \eqref{poi2}  holds as well,
 i.e.,  \eqref{poi3} implies \eqref{poi2}. 
  \end{remark}
  
As we see from \eqref{f7ur76}, the density $\frac{d\pi_\rho}{d\pi_\lambda}$ has a rather complicated form.
 This is why we will not use Theorem \ref{vufr76} in this paper. Instead, we will use the following stronger condition on $\phi$ to get a much  simpler form of $\frac{d\pi_\rho}{d\pi_\lambda}$.  The following theorem is taken from Takahashi \cite{bib3}.  (In fact, Theorem \ref{thrma}  is used to prove  Theorem \ref{vufr76} in  \cite{bib3}).

\begin{theorem}\label{thrma}
 Let $X$ be a locally compact  Polish space. Let $\rho$ and $\lambda$ be non-atomic Radon measures on $(X,\mathcal B(X))$.
Assume $\lambda$ and $\rho$ are equivalent and denote the density $\phi: = \frac{d\rho}{d\lambda}$. Assume that condition \eqref{poi3} holds. Then $\pi _\rho$ and $\pi_\lambda$ are equivalent and
\begin{equation}\label{poi4}
\frac{d\pi_ \rho}{d\pi_ \lambda}(\gamma)=\exp\left[\langle\gamma ,\log\phi\rangle +\int _X(1-\phi)\,d\lambda\right],
\end{equation}
where $|\log \phi|\in L^1(X,d\gamma)$ for $\pi_\lambda$-a.a.\ $\gamma\in\Gamma(X)$.
\end{theorem}

\begin{remark}\label{tfye6e645}
Note that, in formula \eqref{poi4},
$\exp\left[\langle\gamma ,\log\phi\rangle\right]=\prod_{x\in\gamma}\phi(x)$,
where the infinite product converges.
\end{remark}

Let us now recall the definition of a completely random measure, given by Kingman~\cite{bib0}. 
A {\it completely random measure on $X$} is defined as a random measure $\xi$ on $X$ such that, for any mutually disjoint sets $A_1,\dots,A_n\in \mathcal{B}_0(X)$ ($n\in\mathbb N$, $n\ge2$), the random variables $\xi(A_1),\dots,\xi(A_n)$ are independent.

 The following result is obtained by Kingman \cite{bib0}.
  Below we will use the notation $\mathbb R_+:=(0,\infty)$.

\begin{theorem}\label{thrm 2} {\rm (i)}
 Let $\xi_d\in\mathbb{M}(X)$ be a nonatomic Radon measure. Let  a set $\{x_n\}_{n\geq 1}\subset X$ be at most countable. Let $(a_n)_{n\geq 1}$ be a collection of independent,
 nonnegative-valued random variables such that
  \begin{equation}\label{tyr7}
  \text{for each $A\in\mathcal{B}_0(X)$:}\quad\xi_a(A):=\sum_{n=1}^{\infty}a_n\delta_{x_n}(A)<\infty \quad \text{a.s.}\end{equation}
 Let $m$ be a measure on $X\times\mathbb R_+$ such that
  \begin{equation}\label{tyr75ire}
  m(\{x\}\times\mathbb{R}_{+})=0\quad \text{for each $x\in X$,}\end{equation}
   and
    \begin{equation}\label{cmpl4.1}
 \int_{A\times\mathbb{R}_{+}}\min\{s,1\}\,dm(x,s)<+\infty\quad
\text{for each }A\in\mathcal B_0(X)
 \end{equation}
 Let $N$ be a Poisson point process on $X\times\mathbb R_+$ with intensity measure $m$. Assume that $N$ is independent of the random variables $(a_n)_{n\ge1}$.
 Define a random measure
$$\xi_r(A):=\int_{A\times\mathbb{R}_+}s\,dN(x,s).$$
Then $\xi_d$, $\xi_a$, $\xi_r$ are independent, completely random measures on $X$. Furthermore,
$\xi=\xi_d+\xi_a+\xi_r$ is a also a completely random measure on $X$.

{\rm (ii)} Let $\xi$ be a completely random measure on $X$. Then there exist independent, completely random  measures $\xi_d$, $\xi_a$, $\xi_r$ as in part {\rm (i)} such that
$\xi=\xi_d+\xi_a+\xi_r$.
 \end{theorem}

\begin{remark}
\eqref{cmpl4.1} is equivalent to
\begin{align}
\int_{A\times(0,1]}s\,dm(x,s)&<\infty,\label{ftr756}\\
\int_{A\times[1,+\infty)}dm(x,s)&<\infty\label{fty6e654}
\end{align}
for each $A\in\mathcal B_0(X)$.
\end{remark}

 \begin{remark} In fact, Kingman \cite{bib0} (see also \cite{bib01}) does not assume that a random measure takes values in the space of Radon measures.
 He allows a random measure to take values in the space of {\it all} measures on $(X,\mathcal B(X))$ and assumes that, for each $A\in\mathcal B(X)$, $\xi(A)$ is a random variable (i.e., a measurable mapping.) In that case, one does not need condition \eqref{tyr7} to hold.
  However, Daley and Vere-Jones \cite[Theorem  6.3.VIII]{bib} do assume that a random measure takes values in the space of Radon measures, but they do not assume \eqref{tyr7}. It is clear that, without this condition, a measure $\xi_a$ may not be a Radon measure (even possibly a.s.) So, Theorem~\ref{thrm 2} is a refinement of  \cite[Theorem  6.3.VIII]{bib}.

 \end{remark}

In this paper, we will only use part (i) of Theorem~\ref{thrm 2}. For the reader's convenience and for our references below, we will now present the proof of part (i)  and we will also discuss in detail the construction of the completely random measure $\xi_r$, cf.\  \cite{HKPR}, Section~3 in \cite{HKEV}, and  subsection~2.2 in \cite{ba1}. 

\begin{proof}[Proof of part (i) of Theorem~\ref{thrm 2}] 
 Since the  measure $\xi_d $ is deterministic, it is trivially a completely random measure.

Next, we need to prove that $\xi_a(A)$ is  a completely random measure. By the definition of $\xi_a(A)$, for each $A\in \mathcal{B}_0(X)$, we have that
$$\xi_a(A):=\sum_{k=1}^\infty a_k\delta_{x_k}(A)=\sum_{k\geq 1,\ x_k\in A}a_k.$$
If sets $A_1,\dots,A_n\in\mathcal B_0(X)$ are disjoint, then
the random variables $a_k$ appearing in each sum
$\xi_a(A_i)=\sum_{k\geq 1,\ x_k\in A_i}a_k$
are different, so   $\xi_a(A_1),\dots,\xi_a(A_n)$ are independent random variables. Furthermore, \eqref{tyr7} ensures that $\xi_a(A)$ is a Radon measure a.s. Thus, $\xi_a$ is a completely random measure.

Now, we need to prove that $\xi_r$ a completely random measure. Consider the product space  $\hat X:=X\times\mathbb{R}_+$ where $\mathbb{R}_+:=(0,+\infty)$. We need to make $\mathbb{R}_+$ a locally compact Polish space. Consider the bijective mapping
$\mathbb{R}\ni x\mapsto e^x\in\mathbb{R}_+$.
Its inverse mapping is the logarithm function $\ln(x)$. For $s_1, s_2 \in \mathbb{R}_+$, we then take the distance between them in $\mathbb{R}_+$ as the usual distance in $\mathbb{R}$ between $\ln(s_1)$ and $\ln(s_2)$. Thus,
$$\operatorname{dist}  (s_1,s_2)=|\ln  s_1-\ln  s_2|=\left|\ln\left(\dfrac {s_1} {s_2}\right)\right|.$$
Equipped with this metric, $\mathbb{R}_+$ is a locally compact Polish space. Taking the product of $X$ and $\mathbb{R}_+$, we obtain a locally compact Polish space $\hat X$.
The Borel $\sigma$-algebra on $\hat X$ is denoted by $\mathcal B(\hat X)$.

Next,  on the space $\hat{X}$ we want to construct a Poisson point process with intensity measure $m$.
To this end, we should prove that $m$ is a Radon measure on $\hat X$. It suffices to prove that, for each $A\in\mathcal B_0(X)$ and each closed interval $[a,b]\subset\mathbb R_+$,
$m(A\times[a,b])<\infty$.
In fact, we will prove that, for each $A\in\mathcal B_0(X)$ and $\epsilon>0$,
\begin{equation}\label{gur7ir}
m(A\times[\epsilon,\infty))<\infty.\end{equation}
By \eqref{ftr756},  for each $0<\epsilon\le 1$, we have that
\begin{equation}
\int_{A\times [\epsilon , 1]}dm(x,s) \le\int_{A\times [\epsilon , 1]}\frac s \epsilon\, dm(x,s)\le \dfrac 1 \epsilon\int_{A\times(0,1]}s\,dm(s)< +\infty.\label{VM}\end{equation}
Hence, by \eqref{fty6e654}, this implies \eqref{gur7ir}.

By \eqref{tyr75ire}, the Radon measure $m$ is nonatomic. Hence,
 we can construct  $\pi_m$,  the Poisson measure on $(\Gamma (\hat{X}),\mathcal{B}(\Gamma (\hat{X})))$ with intensity measure $m$.

  Let $\Gamma_p(\hat{X})$ denote the set of all pinpointing configurations in $\hat{X}$: \begin{equation*}
\Gamma_p(\hat{X}):=\big\{\gamma\in \Gamma(\hat{X})\mid\text{if }(x_1, s_1), (x_2, s_2)\in \gamma, (x_1, s_1)\neq(x_2, s_2), \text{ then } x_1 \neq x_2\big\}.
\end{equation*}
It is known that
\begin{equation}
\label{hyu43y}
 \Gamma_p(\hat{X})\in\mathcal{B}(\Gamma(\hat{X})),
\end{equation} see \cite{HKPR}.

By \eqref{tyr75ire} and the explicit construction of Poisson measure in a finite volume (see e.g.\ \cite{bib01}), we conclude that 
\begin{equation}\label{vmur 1}
\pi_{m}(\Gamma_p(\hat{X}))=1,
\end{equation}
i.e.,   the Poisson measure $\pi_m$ is concentrated on the set of pinpointing configurations.

Now for each $\gamma\in\Gamma_p(\hat{X})$ and $A\in \mathcal{B}_0(X)$, we define a local mass by
$$
\mathfrak{M}_A(\gamma):=\sum _{(x,s)\in \gamma}\chi
_A (x)s
=\int _{\hat{X}}\chi _A(x)s\,d\gamma(x,s)\in [0, +\infty].
$$
We then define the set of pinpointing configurations with finite local mass by
$$\Gamma _{pf}(\hat{X}):=\big\{\gamma\in\Gamma(\hat{X})\mid \mathfrak{M}_A(\gamma)<\infty \text{ for each }A\in\mathcal{B}_0(X)\big\}.$$
\begin{lemma} We have 
$\pi_m(\Gamma _{pf}(\hat{X}))=1$.
\end{lemma}

\begin{proof}  Let $A\in\mathcal{B}_0(X)$. By condition \eqref{ftr756} and  the Mecke identity (e.g.\ \cite{Mec67}),
$$
\int_{\Gamma(\hat X)}\sum_{(x,s)\in\gamma}\chi _A(x)\chi _{(0,1]}(s)s\,d\pi_m(\gamma)=\int _{A\times(0,1]}s\,dm(x,s)<+\infty.$$
 Hence, 
\begin{equation}\label{ftyr76idr}
\sum_{(x,s)\in\gamma}\chi _A(x)\chi _{(0,1]}(s)s<+\infty \quad \text{for $\pi_m$-a.a.\ }\gamma\in\Gamma_p(X).\end{equation}

By condition  \eqref{fty6e654} and construction of the Poisson measure, 
$$|\gamma\cap (A\times(1,+\infty))|<\infty \quad \text{for $\pi_m$-a.a.\ }\gamma\in\Gamma_p(X).$$
This implies
\begin{equation}\label{trer6e5}
\sum_{(x,s)\in\gamma}\chi_A(x)\chi_{(1,+\infty)}(s)s<+\infty \quad \text{for $\pi_m$-a.a.\ }\gamma\in\Gamma_p(X).\end{equation}
Note that $X$ can be represented as a countable union of compact sets.
Hence, the lemma follows.
\end{proof}

Next, we define on $X$ the set of discrete Radon measures:
$$\mathbb{K}(X):=\left\{\eta=\sum_i s_i\delta_{x_i}\in \mathbb{M}(X)\mid s_i>0, x_i\in X\right\}.$$
Here, $\delta_{x_i}$ is the Dirac measure with mass at $x_i$, the atoms $x_i$ are assumed to be distinct and their total number is at most countable. By convention, the cone $\mathbb K(X)$ contains the null mass $\eta=0$, which is represented by the sum over the empty set of indices $i$. We denote $\tau(\eta):=\{x_i\}$, i.e., the set on which the measure $\eta$ is concentrated. For $\eta\in\mathbb K(X)$ and $x\in\tau(\eta)$, we  denote by $s_x$ the mass of $\eta$ at  point $x$, i.e., $s_x:=\eta(\{x\})$. Thus, each $\eta\in\mathbb K(X)$ can be written in the form $\eta=\sum_{x\in\tau(\eta)}s_x\delta_x$.

Note that the closure of $\mathbb K(X)$ in the vague topology coincides with $\mathbb M(X)$.
As shown in \cite{HKPR}, $\mathbb K(X)\in\mathcal B(\mathbb M(X))$. We denote by $\mathcal B(\mathbb K(X))$ the trace $\sigma$-algebra of $\mathcal B(\mathbb M(X))$ on $\mathbb K(X)$.

 Let us now construct a bijective mapping
 \begin{equation}\label{R1}
 \mathcal R:\Gamma_{pf}(\hat X)\to \mathbb{K}(X)
 \end{equation}
 as follows: For each $\gamma=\{(x_i,s_i)\}\in\Gamma_{pf}(\hat X)$, we set
 \begin{equation}\label{R2}
  \mathcal R\gamma:=\sum_i s_i\delta_{x_i}\in\mathbb K(X).
 \end{equation}
 By \cite[Theorem~6.2]{HKPR}, we have
 
 \begin{equation}\label{hyutw3r4}
 \mathcal B(\mathbb K(X))=\{\mathcal R A\mid A\in\mathcal B(\Gamma_{pf}(\hat X))\}.\end{equation}
 Hence, both $\mathcal R$ and $\mathcal R^{-1}$ are measurable mappings.

  Let $\xi_r$ be the pushforward of $\pi_m$ under $\mathcal{R}:\Gamma_{pf}(\hat{X})\mapsto \mathbb{K}(X)$. 
  If $A_1,\dots,A_n\in\mathcal{B}_0(X)$ are mutually disjoint, then $\gamma(B_1),\dots,\gamma(B_n)$ are independent random variables under $\pi_m$ if $B_1\subset A_1\times\mathbb{R}_+,\dots,B_n\subset A_n\times\mathbb{R}_+$. Therefore, the random variables
 $$\int_{\hat{X}}\chi_{A_1}(x)s\,d\gamma(x,s),\dots,\int_{\hat{X}}\chi_{A_n}(x)s\,d\gamma(x,s)$$ are independent under $\pi_m$. This implies that $\eta(A_1),\dots,\eta(A_n)$ are independent under $\xi _r.$ Thus, $\xi_r $ is a completely random measure.

Trivially, the sum $\xi_d+\xi_a+\xi_r$ is a completely random measure as well. Thus, part (i) Theorem~\ref{thrm 2} is proven.
 \end{proof}
 
 The following result is immediate now.
 
 \begin{corollary}\label{mygfrcor45}
 Let $m$ be a measure on $X\times \mathbb{R}_+$ which satisfies \eqref{tyr75ire} and \eqref{cmpl4.1}. Then there exists a completely random measure $\mu _m$ such that $\mu _m(\mathbb{K}(X))=1$ and which has Fourier transform
 \begin{equation}\label{nwm1}
 \int_{\mathbb{K}(X)}e^{i\langle f , \eta\rangle}d\mu _m(\eta)=\exp\left[
\int _X\int_{\mathbb{R}_+}\left (e^{isf(x)}-1\right) dm(x,s)\right],\quad f\in C_0(X).
 \end{equation}
The measure $m$ will be called the L\'evy measure of the completely random measure $\mu_m$.
 \end{corollary}
 
\begin{remark}
It is easy to see that \eqref{nwm1} remains true if $f\in B_0(X)$, i.e., $f:X\to \mathbb{R}$  is a measurable bounded function with compact support. In particular, for any $A\in\mathcal{B}_0(X)$ and $t\in \mathbb{R}$, we may take $f(x)=t\chi_A(x)$.
 Then by \eqref{nwm1}
 $$\int_{\mathbb{K}(X)}e^{it\eta(A)}d\mu _m(\eta)=\exp\left[\int_A\int_{\mathbb{R}_+}\left(e^{ist}-1\right)dm(x,s)\right].$$
 In particular, if $m$ is product measure:
 $$dm(x,s)=d\sigma(x)d\lambda(s),$$
 then
 $$\int_{\mathbb{K}(X)}e^{it\eta(A)}d\mu_m(\eta)=\exp\left[\sigma(A)\int_{\mathbb{R}_+}\left(e^{ist}-1\right)d\lambda(s)\right].$$
 Thus, in this case the distribution of the random variable $\eta(A)$ only depends on $\sigma(A)$. This is why in such a case, one calls $\mu _m$ a \textit{measure-valued L\'{e}vy processes.}
 \end{remark}

The corollary below follows immediately from Theorem~\ref{thrm 2} and its proof.

\begin{corollary} Let $\xi$ be a completely random measure on $X$. Then there exist a deterministic, nonatomic Radon measure $\xi_d$ and completely random measure $\xi'$, taking values a.s.\ in the space $\mathbb K(X)$ of discrete Radon measures on $X$, such that $\xi=\xi_d+\xi'$.
\end{corollary}

A completely random measure on $X$ which takes a.s.\ values in  $\mathbb K(X)$ is called a {\it completely random discrete measure}. In particular, the measure $\xi_r$ from Theorem~\ref{thrm 2} is a {\it completely random discrete measure without fixed atoms}. Below we will only be interested in such completely random measures.

 \section{Quasi-invariance 
of completely random measures with respect to transformations of weights}\label{cte65i}

In this section, we will consider the current group which transforms the weights.  Let $\sigma$ be a fixed Radon non-atomic measure on $(X,\mathcal B(X))$.
 
  \subsection{General theory}
  
  We define 	
 \begin{multline*}
  C_0(X\to\mathbb{R}_+):=\{\theta:X\to \mathbb R_+\mid \theta \text{ is continuous and }\\
   \theta=1 \text{ outside a compact set in } X\}.\end{multline*}
 $C_0(X\to\mathbb{R}_+)$ is a (commutative) group under the usual point-wise multiplication of functions. In particular, the identity element in this group is the function which is identically equal to $1$ on $X$.
We call $C_0(X\to\mathbb{R}_+)$ a {\it current group}.

 We define the action of the group $C_0(X\to \mathbb{R}_+)$ on $\mathbb{M}(X)$ (the set of Radon measures) by
 $$\mathbb{M}(X)\ni \eta\mapsto\theta \eta\in \mathbb{M}(X) \text{ for each }\theta\in C_0(X\to\mathbb{R}_+).$$
 Here $\theta\eta$ denotes the measure on $X$ which has density $\theta$ with respect to the measure $\eta$.

  Assume $\mu_m$ is a completely random measure on $X$ which has Fourier transform \eqref{nwm1}. We are interested whether $\mu_m$ is quasi-invariant with respect to the action of the group $C_0(X\to\mathbb{R}_+)$ on $\mathbb{M}(X)$.
    
  Let  us assume that
 \begin{equation}\label{Q1}
 dm(x,s)=\frac{l(x,s)}{s}\,d\sigma(x)\,ds,
 \end{equation}
 where
 \begin{equation}\label{Q2}
\text{for each $x\in X$, either $l(x,s)>0$ for all $s\in\mathbb R_+$ or $l(x,s)=0$ for all $s\in\mathbb R_+$.} 
 \end{equation}

Below, for a set $Y\in\mathcal B(X)$, we denote by $\mathcal B(Y)$ the  trace $\sigma$-algebra of $\mathcal B(X)$ on $Y$, i.e., the collection of all $A\in\mathcal B(X)$ satisfying $A\subset Y$. We will also denote by $\mathcal B_0(Y)$ the collection of all $A\in\mathcal B(X)$ which satisfy $A\subset Y$.

Let
\begin{equation}\label{rtew64u3u}
Y:=\{x\in X\mid l(x,\cdot)>0\}.\end{equation}
Then, under \eqref{Q1} and \eqref{Q2},   condition \eqref{cmpl4.1} becomes
  \begin{equation}\label{Q3}
 \int_{A\times \mathbb{R}_+}l(x,s)\min\{s^{-1},1\}\,d\sigma(x)\,ds<+\infty\quad \text{for all }A\in \mathcal{B}_0(Y).
 \end{equation}
 Note also that that condition \eqref{tyr75ire} is now satisfied.

 The following theorem and Corollary~\ref{cor2} below are the main result of this section. They extend  Theorem 4 and Corollary 5 in \cite{ba1}, proved for measure-valued L\'evy processes.

 \begin{theorem}\label{hout8r6ode}
 Assume \eqref{Q1}, \eqref{Q2} and \eqref{Q3} hold. 
  Assume that, for each $n\in\mathbb{N}$, there exists $\epsilon > 0$ such that, for each $A\in \mathcal{B}_0(Y)$, 
 \begin{equation}\label{Q8}
   \int_A\left[\sup_{r\in\left[\frac 1 n,n\right]}\int_{(0,\epsilon)}\frac{|l(x,rs)-l(x,s)|}{s}\,ds\right]d\sigma(x)<\infty.
   \end{equation}
 Then the measure $\mu_m$ is quasi-invariant with respect to all transformations from the group of currents, $C_0(X\to\mathbb R_+)$, i.e.,  each $\theta\in C_0(X\to\mathbb{R}_+)$ maps $\mathbb{K}(X)$ into itself, and $\mu^\theta _m$ is equivalent  to $\mu _m$. Furthermore, the corresponding density is given by
 \begin{multline}\label{Q5}
 \frac{d\mu_m^\theta}{d\mu _m}(\eta)=\exp\bigg[\int _{Y}\log\left(\frac{l(x,\theta^{-1}(x)s_x)}{l(x,s_x)}\right)s_x^{-1}d\eta(x)
 \\
 \text{}+\int_Y\int_{\mathbb{R}_+}\frac {\left(l(x,s)-l(x,\theta^{-1}(x)s)\right)} s \,ds\,d\sigma(x)\bigg].
 \end{multline}
     In \eqref{Q5}, the function appearing under the sign of integral with respect to measure $\eta$ belongs to $L^1(Y,\eta)$ for $\mu _m$-a.a. $\eta\in\mathbb{K}(X)$.
   \end{theorem}

   \begin{proof} 
 We divide the proof of this theorem into several steps.
 
 {\it Step 1.}  
   Let us first prove that, for each $\theta\in C_0(X\to\mathbb R_+)$,
   \begin{align}
   &\int_Y\int_{\mathbb{R}_+}\frac {\left|l(x,s)-l(x,\theta^{-1}(x)s)\right|} s \ ds \ d\sigma(x)\notag\\
   &\quad= \int_X\int_{\mathbb{R}_+}\frac {\left|l(x,s)-l(x,\theta^{-1}(x)s)\right|} s \ ds \, d\sigma(x)<\infty.\label{tur678}\end{align}
The function $\theta$ is continuous and takes values in $\mathbb{R}_+$. By the definition of $C_0(X\to \mathbb{R}_+)$, there exists a compact set $C\subset X$ such that $\theta(x)=1$ for al $x\notin C$. The function $\theta$ is continuous on the compact set $C$. Hence $\theta$ attains its infimum and supremum on ${C}$. Thus,
   $$\inf_{x\in C}\theta(x)>0, \qquad \sup_{x\in C}\theta(x)<+\infty.$$
   But this implies that, for all $y\in X$,
   $$0<\inf_{x\in X}\theta(x)\le \theta(y)\le\sup_{x\in  X}\theta(x)<\infty.$$
      Hence, there exists $n\in \mathbb{N} $ such that, for all $x\in X$,
   $$\frac 1 n\le \theta(x)\le n.$$
   So, fix this $n\in \mathbb{N},$ and choose the corresponding $\epsilon > 0$ as in the formulation of the theorem. 
 Denote $A=C\cap Y$, $A\in\mathcal B_0(Y)$.  
   We have
   \begin{align*}
  \int_{A\times\mathbb{R}_+}\frac {\left|l(x,\theta^{-1}(x)s)-l(x,s)\right|}s\,d\sigma(x)\,ds= &\int_{A\times (0,\epsilon)}\frac{|l(x,\theta^{-1}(x)s)-l(x,s)|}{s}\,d\sigma(x)\,ds\\
   +&\int_{A\times [\epsilon, +\infty)}\frac{|l(x,\theta^{-1}(x)s)-l(x,s)|}{s}\,d\sigma(x)\,ds.
   \end{align*}
   To prove the finiteness of the first integral, we have,
    for a fixed $x\in A$,
      \begin{equation*}
  \int_{(0,\epsilon)}\frac{|l(x,\theta^{-1}(x)s)-l(x,s)|}{s}ds \le \sup _{r\in\left[\frac 1 n, n\right]}\int_{(0,\epsilon)}\frac{|l(x,rs)-l(x,s)|}{s}ds.
  \end{equation*}
  Hence, by \eqref{Q8},
  \begin{align*}
  &\int _{A\times(0,\epsilon)} \frac{|l(x,\theta^{-1}(x)s)-l(x,s)|}{s}\,d\sigma(x)\,ds \\
  &\qquad
   \le\int_A\left[\sup_{r\in\left[\frac 1 n, n\right]}\int_{(0,\epsilon)}\frac{|l(x,rs)-l(x,s)|}{s}ds\right]d\sigma(x)<+\infty.
  \end{align*}
  
  For the second integral, we have
   \begin{align}\nonumber
   &\int _{A\times[\epsilon,+\infty)}\frac{|l(x,\theta^{-1}(x)s)-l(x,s)|}{s}\,ds\,d\sigma(x)\\ \nonumber
   &\quad\le\frac 1 \epsilon\int_{A\times[\epsilon, +\infty)}|l(x,\theta^{-1}(x)s)-l(x,s)|
  \, d\sigma(x)\,ds\\
    &\quad\le \frac 1 \epsilon\left[\int_{A\times[\epsilon, +\infty)}l(x,\theta^{-1}(x)s)\,d\sigma(x)\,ds+\int_{A\times[\epsilon, +\infty)}l(x,s)\,d\sigma(x)\,ds\right]. \label{Q10}
     \end{align}
          By \eqref{Q3} the second integral in \eqref{Q10} is finite. Let us consider the first integral
     $$\int_{A\times[\epsilon,+\infty)}l(x,\theta^{-1}(x)s)\, d\sigma(x)\,ds.$$
      Let $G$ denote the image of $A\times[\epsilon,+\infty)$ under the mapping $(x,s)\to(x,\theta^{-1}(x)s)$. Then, as $\frac 1 n \le \theta(x)\le n$, we obtain from \eqref{Q3}:
     \begin{align*}
     &\int_{A\times[\epsilon,+\infty)}l(x,\theta^{-1}(x)s)\,d\sigma(x)\,ds\\
     &\quad =\int_G l(x,s)\theta(x)\ d\sigma(x)\ ds\\
     &\quad\le n\int_G l(x,s) \ d\sigma(x)\ ds\\
     &\quad\le n\int_{A\times\left[\frac \epsilon n, +\infty\right)}l(x,s)ds<\infty.
    \end{align*}
    
    Thus,
    \begin{equation}\label{m1}
    \int_A\int_{\mathbb{R}_+}\frac {|l(x,s)-l(x,\theta^{-1}(x)s)|}s\ ds\ d\sigma(x)<\infty.
    \end{equation}
   If $x\notin A$, then either $\theta(x)=1$ or $l(x,s)=0$ for all $s\in\mathbb R_+$. Hence
   $$l(x,s)-l(x,\theta^{-1}(x)s)=0.$$
   Therefore the integral in \eqref{m1} is equal to
   $$\int _X\int _{\mathbb R_+}\frac {|l(x,s)-l(x,\theta^{-1}(x) s)|}s \ ds \ d\sigma(x).$$
   Thus \eqref{tur678} holds.

{\it Step 2.} We will now bring the problem of equivalence 
of the measures $\mu_m$ and $\mu_m^\theta$ to the configuration space $\Gamma_{pf}(\hat X)$.

  Recall that the measure $\mu_m$ was constructed as the pushforward of the Poisson measure $\pi _m$ under the bijective mapping $\mathcal{R}$, see \eqref{R1} and \eqref{R2}. Consider the inverse mapping $$\mathcal{R}^{-1}:\mathbb{K}(X)\to \Gamma_{pf}(\hat{X}),$$
    with
    $$\mathcal{R}^{-1}\left(\sum_i s_i\delta_{x_i}\right) = \{(x_i, s_i)\}.$$
    As we already know $\mathcal{R}^{-1} $ is measurable. Denote by $\pi _m^\theta$ the pushforward of $\mu_m^\theta$ under $\mathcal{R}^{-1}.$ Note that
      \begin{align}\nonumber
    & \mathcal{R}^{-1}\theta  \mathcal{R}:\Gamma_{pf}(\hat{X})\to \Gamma_{pf}(\hat{X}), \text { and }\\
    & \gamma=\{(x_i,s_i)\}\to\{(x_i,\theta(x_i) s_i)\}.\label{m3}
    \end{align}
    Hence, $\pi_m^\theta$ is the pushforward of the measure $\pi_m$ under the transformation \eqref{m3}. Thus, for each $f\in C_0(X\times\mathbb{R}_+)$ and $\gamma=\{(x_i,s_i)\}\in\Gamma_{pf}(\hat{X}),$
    $$
    \langle f,\mathcal{R}^{-1}\theta\mathcal{R}\gamma\rangle=\sum _i f(x_i,\theta(x_i)s_i)
    =\langle f^\theta,\gamma\rangle,
    $$
    where $f^\theta:X\times\mathbb{R}_+\to\mathbb{R}$ and
      $f^\theta (x,s)=f(x,\theta(x)s).$
        Hence, the Fourier transform of $\pi_m^\theta$ is
      \begin{align*}
      \int_{\Gamma_{pf}(\hat{X})}e^{i\langle f,\gamma\rangle}d\pi_m^\theta(\gamma)=&\int_{\Gamma_{pf}(\hat{X})}e^{\langle f^\theta,\gamma\rangle}d\pi _m(\gamma)\\
      =&\exp\left [\int_X\int_{\mathbb{R}_+}(e^{if(x,\theta(x)s)}-1)\frac{l(x,s)}{s} \ ds \ d\sigma(x) \right]\\ 
      &=\exp\left[\int_X\int_{\mathbb{R}_+}(e^{if(x,s)}-1)\frac{l(x,\theta^{-1}(x)s)}{s}\, {d\sigma(x)} \ {ds}\right].
\end{align*}
Hence, $\pi_m^\theta$ is the Poisson measure on $\Gamma_{pf}(\hat{X})$ with intensity measure
$$dm^\theta(x,s):=\frac{l(x,\theta^{-1}(x)s)}{s} \, d\sigma(x) \ ds.$$

Thus, to prove that the measures $\mu _m$ and $\mu _m^\theta$ are equivalent, it is sufficient to prove that the measures $\pi _m$ and $\pi _m^\theta$ are equivalent. 

{\it Step 3.} By using Theorem \ref{thrma}, we will now show that the measures $\pi _m$ and $\pi _m^\theta$ are equivalent.

By \eqref{Q2}, both measures $m$ and $m^\theta$ are concentrated on $Y\times\mathbb R_+$, are equivalent and 
\begin{equation*}
\frac{dm^\theta}{dm}(x,s)=\frac{l(x,\theta^{-1}(x)s)s}{s \ l(x,s)}\,\chi_Y(x) = \frac {l(x,\theta^{-1}(x)s)}{l(x,s)}\,\chi_Y(x) .
\end{equation*}
We have by \eqref{tur678},
\begin{align*}
&\int _{Y\times\mathbb{R}_+} \left | \frac {l(x,\theta^{-1}(x)s)}{l(x,s)}-1\right |dm(x,s)\\
& \quad =\int_{Y\times\mathbb{R}_+}\frac {|l(x,\theta^{-1}(x)s)-l(x,s)|} s \ d\sigma(x) \ ds < \infty.
\end{align*}
Hence, by Theorem \ref{thrma}, the measures $\pi_m$ and $\pi _m ^\theta$ are equivalent, hence so are $\mu_m$ and $\mu_m^\theta$.

Also by Theorem \ref{thrma}, for $\gamma=\{(x_i,s_i)\}\in \Gamma_{pf}(\hat{X})$,
\begin{align*}
\frac {d\pi _m ^\theta}{d\pi_m}(\gamma)=&\exp\left[\left\langle\log\left(\frac{dm^\theta}{dm}\right )\chi_{\hat Y}, \gamma\right\rangle +\int _{\hat Y}\left(1-\frac {dm^\theta}{dm}\right)dm \right]\\
=&\exp \bigg[\sum _i \log\left(\frac{l(x_i,\theta^{-1}(x_i)s_i)}{l(x_i,s_i)}\right)\chi_{\hat Y}(x_i,s_i)\\
&\qquad+ \int_{\hat{Y}}\left (1- \frac{l(x,\theta^{-1}(x)s)}
{l(x,s)}\right)\frac {l(x,s)}{s}\ d\sigma(x) \ ds\bigg]\\
=& \exp \bigg[\sum _i \log\left(\frac{l(x_i,\theta^{-1}(x_i)s_i)}{l(x_i,s_i)}\right)\frac{s_i}{s_i}\, \chi_{\hat Y}(x_i,s_i)\\
&\qquad + \int_{\hat Y} (l(x,s)-l(x,\theta^{-1}(x)s) \, \frac {1} s \, d\sigma(x) \ ds\bigg],
\end{align*}
where $\hat Y=Y\times\mathbb R_+$.
From here formula \eqref{Q5} follows.
\end{proof}
\begin{corollary}
Assume that  the condition of Theorem \ref{hout8r6ode} hold. For each $\theta\in C_0(X\to \mathbb{R}_+)$, we define a unitary operator $\mathscr {U}_\theta $ in $L^2(\mathbb{K}(X)\to \mathbb{C},\mu _m)$ by $$(\mathscr{U}_\theta f)(\eta)=f(\theta^{-1}\eta)\sqrt{\frac{d\mu^\theta_m}{d\mu_m}(\eta)},$$
where the Radon--Nikodym density $\frac{ d\mu_m
^\theta}{d\mu_m}$ is given by \eqref{Q5}. Then the operators $\mathscr{U}_\theta$, $\theta\in C_0(X\to\mathbb{R}_+)$, form a unitary representation of the current group $C_0(X\to\mathbb{R}_+)$.
\end{corollary}

\begin{corollary}\label{cor2}
Assume \eqref{Q1}--\eqref{Q3} are satisfied. Assume that, for some $\epsilon>0$,
$$ l(x,s)=l_1(x,s)+l_2(x,s)\quad \text{for $x\in Y$, $s\in (0,\epsilon)$},$$ 
where $Y$ is defined by \eqref{rtew64u3u}. 
 Here, for each  fixed $x\in Y$, the function $l_1(x,s)$ is differentiable in $s$ on $(0,\epsilon)$, and for each $n\in \mathbb N$ and $A\in\mathcal B_0(Y)$, 
\begin{equation}\label{mmu1}
\int_A\int_{(0,\frac \epsilon n)}\sup_{u\in[\frac s n,sn]}\left|\frac \partial{\partial u}l_1(x,u)\right|ds\,d\sigma(x)<\infty
\end{equation}
and
\begin{equation}\label{mmu2}
\int _A\int_{(0,\epsilon)}\frac{l_2(x,s)}{s}\,ds\,d\sigma(x)< \infty.
\end{equation}
Then condition \eqref{Q8} is satisfied, and so the conclusion of Theorem~\ref{hout8r6ode} holds.
\end{corollary}

\begin{proof}
Using that $l(x,s)=l_1(x,s)+l_2(x,s)$, we get
\begin{align*}
&\sup_{r\in\left[\frac 1 n,n\right]}\int_{(0,\epsilon)}\frac{|l(x,rs)-l(x,s)|}{s}\,ds\\
&\le\sup_{r\in\left[\frac 1 n,n\right]}\int_{(0,\epsilon)}\frac{|
l_1(x,rs)-l_1(x,s)|}s\,ds+\sup_{r\in\left[\frac 1 n,n\right]}\int_{(0,\epsilon)}\frac{|l_2(x,rs)-l_2(x,s)|}s\,ds.
\end{align*}
Hence, it suffices to prove that \eqref{Q8} holds for both $l(x,s)=l_1(x,s)$ and for $l(x,s)=l_2(x,s)$. By Taylor's formula,
\begin{equation*} 
|l_1(x,rs)-l_1(x,s)|=\left|\frac \partial{\partial u}l_1(x,u)\Big|_{u=u_o}\right|\,|rs-s|,\end{equation*}
where $u_0$ is a point between $rs$ and $s$, that is for $r<1,$ $u_0\in(rs,s)$ and $r>1$, $u_0\in(s,rs)$. Therefore, for $r\in\left[\frac 1 n,n\right]$, we have $u_0\in\left[\frac s n,sn\right]$. Hence, for $r\in\left[\frac 1 n,n\right]$,
\begin{equation*}
|l_1(x,rs)-l_1(x,s)|\le\sup_{u\in\left[\frac s n,sn\right]}\left|\frac\partial{\partial u}l_1(x,u)\right|ns.\end{equation*}
This implies, by \eqref{mmu1},
\begin{align*}
\int_A\Bigg[\sup_{r\in\left[\frac 1n,n\right]}&\int_{(0,\epsilon)}\frac{|l_1(x,rs)-l_1(s,s)|}s\,ds\Bigg]d\sigma(x)\\
\le \int _A &\left[\sup_{r\in\left[\frac 1 n,n\right]}\int_{(0,\epsilon)}\sup_{u\in\left[\frac sn, n\right]}\left|\frac \partial {\partial u}l_1(x,u)\right|n\,ds\right]d\sigma(x)\\
&=n\int_A\int
_{(0,\epsilon)}\sup_{u\in \left[\frac sn,sn\right]}\left|\frac{\partial}{\partial u}l_1(x,u)\right|ds\,d\sigma(x)<\infty,
\end{align*}
where $A\in\mathcal{B}_0(Y)$. Thus, the statement is proven for $l_1$.

Now, let us prove the statement for $l_2$. For $r\in\left[\frac 1 n,n\right]$, and $A\in\mathcal{B}_0(Y)$,
\begin{align*}
&\int_A\Bigg[\sup_{r\in\left[\frac 1 n,n\right]}\int_{(0,\epsilon/n)}\frac{|l_2(x,rs)-l_2(x,s)|}s\,ds\Bigg]d\sigma(x)\\
&\quad\le\int_A\Bigg[\sup_{r\in\left[\frac 1 n,n\right]}\int_{(0,\epsilon/n)}\frac{l_2(x,rs)}s\,ds\Bigg]d\sigma(x)+\int_A\left[\sup_{r\in\left[\frac 1 n,n\right]}\int_{(0,\epsilon/n)}\frac{l_2(x,s)}s\,ds\right]d\sigma(x)\\
&\quad=\int_A\left[\sup_{r\in\left[\frac 1 n,n\right]}\int_{(0,\epsilon/n)}\frac {l_2(x,s)}s\,ds\right]d\sigma(x)+\int_A\int_{(0,\epsilon/n)}\frac{l_2(x,s)}s\,ds\,d\sigma(x)\\
&\quad \le 2\int_A\int_{(0,\epsilon)}\frac{l_2(x,s)}s\,ds\,d\sigma(x)<\infty
\end{align*}
by \eqref{mmu2}.
\end{proof}

\subsection{Examples}

We will now consider examples of completely random measures which satisfy the assumptions of Corollary \ref{cor2}.

\subsubsection{Completely random gamma measures}\label{tyd576e}
Let us fix two parameters $\alpha>0$ and $\beta>0$. We first consider the function
$$l(x,s)=l(s)=\beta e^{-\frac s\alpha },$$
so that
\begin{equation}\label{gur76r5}
dm(x,s)=\beta\frac{e^{-\frac s\alpha }}s\,d\sigma(x)\,ds.\end{equation}
Note that
$$dm(x,s)=d\sigma(x)\,d\lambda(s),$$
where
$$d\lambda(s)=\beta\frac{e^{-\frac s\alpha }}s\,ds.$$
Following \cite{bibb3}, we will call the measure $\mu _m$ the {\it gamma measure}, or the {\it measure-valued gamma process with parameters $\alpha$ and $\beta$}.

\begin{proposition}\label{gtyfr67}
	The Laplace transform of the measure $\mu_m$ with $m$ given by \eqref {gur76r5} is
	\begin{equation}\label{lurify1}
\int_{\mathbb K(X)}\exp[-\langle\eta,f\rangle]d\mu_m(\eta)=\exp\left[-\beta\int _X\log(1+\alpha f(x))\, d\sigma(x)\right],\end{equation}
where $f:X\to\R $ is a bounded measurable function with compact support which satisfies  $f(x)> -\frac 1\alpha$ for all $x\in X$.
\end{proposition}

This result is known, see \cite{bibb3}, but we will now give  a complete proof of it, since we will later on need it. 

\begin{proof}
	We start with the following known result. 
	
	\begin{lemma}\label{kinglm1}
For $u>-1$,
\begin{equation}\label{yuuumt6}
\int_0^\infty \frac{e^{-us}-1}{s}\,e^{-s}\,ds=-\log(1+u).
\end{equation}	
	\end{lemma}
	
	
	By Lemma~\ref{kinglm1}, for $u>\frac 1\alpha$,
		\begin{align}
	&\int_0^\infty \frac{e^{-us}-1}{s}\,e^{-\frac s\alpha }\,ds	=\int_0^\infty\frac{e^{-s \alpha u}-1}{s}\,e^{-s}\,ds\notag\\
	&\quad =-\log\left(1+\alpha u\right).\label{buytr687}
	\end{align}

	Using the construction of the measure $\mu_m$ and the Laplace transform of the Poisson measure,  we have
	\begin{equation}
	\label{newm11}
	\int_{\mathbb K(X)}\exp\left[-\langle \eta,f\rangle\right]d\mu_m(\eta)=\exp\left[\int_X\int_{\mathbb{R}_+}(e^{-f(x)s}-1)e^{-\frac s\alpha }\,\frac \beta s \,ds\, d\sigma(x)\right].
	\end{equation}
		By \eqref{buytr687}, for each $x\in X$,
	\begin{equation*}
	\int_{\mathbb R_+}\left(e^{-f(x)s}-1\right)\,e^{-\frac s\alpha }\,\frac 1 s \,ds=-\log\left(1+\alpha f(x)\right).
	\end{equation*}
		Now, substituting the above result into the right hand side of equation  \eqref{newm11}, we get
	 \eqref{lurify1}.
		\end{proof}

Let $\Delta\in\mathcal{B}_0(X)$. 	By \eqref{lurify1}, for each  $t>-\frac1\alpha$,
\begin{align}
\int_{\mathbb K(X)} \exp[-t\eta(\Delta)]&=\exp\left[-\beta\int_X
\log(1+\alpha t\chi_\Delta(x))\,d\sigma(x)\right]\notag\\
&=(1+\alpha t)^{-\beta\vol(\Delta)}.\label{tyd67i5ei}
\end{align}

Let us recall that the gamma distribution on $\mathbb{R}$ with parameters $\alpha$ and $\theta$ is defined by
$$\frac{u^{\theta-1}}{\alpha^\theta \Gamma(\theta)}\,e^{-\frac{u}{\alpha}}\,\chi_{(0,\infty)}(u)\,du.
$$
The Laplace transform of the gamma distribution is given by
$$\int_{\R}e^{-tu}\,\frac{u^{\theta-1}}{\alpha^\theta \Gamma(\theta)}\,e^{-\frac{u}{\alpha}}\,\chi_{(0,\infty)}(u)\,du =(1+\alpha t)^{\theta},\quad t>-\frac1\alpha.$$
Hence,  under $\mu_m$, the random variable $\eta(\Delta)$ has gamma distribution with parameters $\alpha$ and  $\theta=\beta\,\vol(\Delta)$.

Now, we will produce a generalization by making the parameters $\alpha$ and $\beta$ to be positive functions on $X$. Thus, let us consider measurable functions
$$\alpha:X\to \mathbb{R}_+, \quad \beta: X\to[0,\infty).$$
We define
$$l(x,s)=\beta(x)\,e^{-\frac s{\alpha(x)}},$$
so that
\begin{equation}\label{yurt}
dm(x,s)=\frac{\beta(x)e^{-\frac s{\alpha(x)}}}s\,d\sigma(x)\,ds.
\end{equation}
We denote by $L^1_{\mathrm{loc}}(X,\sigma)$ the space of all measurable functions $f:X\to\mathbb{R}$ such that, for each $A\in\mathcal{B}_0(X)$, $f\chi_A\in L^1(X,\sigma)$,
i.e., $\int_A|f(x)|\,d\sigma(x)<\infty$.

\begin{lemma}\label{lmkuy}
Assume that the function $\alpha \beta $ belongs to $L^1_{\mathrm{loc}}(X,\sigma)$. Then the measure $m$ given by \eqref{yurt} satisfies  \eqref{Q3}.
\end{lemma}
\begin{proof}
For each $A\in\mathcal B_0(X)$, we have
\begin{align*}
&\int_{A\times\R_+}l(x,s)\min\{s^{-1},1\}\,d\sigma(x)\,ds\le
\int_{A\times\R_+}l(x,s)\,d\sigma(x)\,ds\\
&\quad=\int_A \int_0^\infty\beta(x)e^{-\frac s{\alpha(x)}}\,ds\,d\sigma(x)\\
&\quad=\int_A\beta(x)\left(\int_0^\infty e^{-\frac s{\alpha(x)}}\,ds\right)\,d\sigma(x)\\
&\quad=\int_A \alpha(x)\beta(x)\,d\sigma(x)<\infty.\qedhere
\end{align*}
\end{proof}

\begin{proposition}
	The Laplace transform of the measure $\mu_m$ with $m$ given by \eqref{yurt} is
	\begin{equation*}
	\int_{\mathbb{K}(X)}\exp[-\langle\eta,f\rangle]d\mu_m(\eta)=\exp\left[-\int_X(1+\alpha(x)f(x))\beta(x)d\sigma(x)\right],
	\end{equation*}
	where $f:X\to\R$ is a bounded, measurable function with compact support which satisfies $\alpha(x)f(x)>-1$ for all $x\in X$.

\end{proposition}
\begin{proof}
	Analogously to \eqref{newm11}, we have
	\begin{equation*}
	\int_{K(X)}\exp[-\langle\eta,f\rangle]d\mu_m(\eta)=\exp\left[\int_X\int_{\mathbb{R}_+}\left(e^{-f(x)s}-1\right)e^{-\frac s{\alpha(x)}}\,\frac{\beta(x)}{s}\,ds\,d\sigma(x)\right].
	\end{equation*}
	By \eqref{buytr687},
	$$
	\int_{\mathbb{R}_+}\left(e^{-f(x)s}-1\right)e^{-\frac s{\alpha(x)}}\,\frac{\beta(x)}{s}\,ds=-\beta(x)\log(1+\alpha(x)f(x)),
	$$
	which implies the proposition.
\end{proof}

\begin{lemma}\label{tudr66dr}
Assume that the functions $\alpha \beta $ and $\beta$ belong to $L^1_{\mathrm{loc}}(X)$. Then  the measure $m$ satisfies the conditions of Corollary \ref{cor2}.\end{lemma}

\begin{proof} Fix any $\epsilon>0$.
In the notations of Corollary \ref{cor2}, we set $l_1=l$ and $l_2=0$. The function $l(x,s)$ is evidently differentiable in the $s$ variable. Thus, we only have to check that, for each $n\in \mathbb N$ and $A\in\mathcal B_0(X)$, $A\subset Y=\{y\in X\mid \beta(y)=0\}$,
$$\int_A\int_0^{\frac \epsilon n}\sup_{u\in\left[\frac s n,sn\right]}\left|\frac \partial{\partial u}l(x,u)\right|ds\,d\sigma(x)<\infty.$$
  We have
$$\left|\frac{\partial}{\partial u}l(x,u)\right|=\frac{\beta(x)}{\alpha(x)}e^{-\frac u{\alpha(x)}},$$
thus
$$\sup_{u\in\left[\frac s n,sn\right]}\left|\frac{\partial}{\partial u}l(x,u)\right|=\frac{\beta(x)}{\alpha(x)}e^{-\frac s{\alpha(x)n}}.$$
We have
\begin{align*}
\int_A\int_0^{\frac \epsilon n}&\frac{\beta(x)}{\alpha(x)}e^{-\frac s{\alpha(x)n}}ds\,d\sigma(x)\\
&=\int_A\frac{\beta(x)}{\alpha(x)}(-n)\alpha(x)\left[e^{-\frac\epsilon{\alpha(x)n^2}}-1\right]\,d\sigma(x)\\
&\le n\int_A\beta(x)d\sigma(x)<\infty.
\end{align*}
Therefore, \eqref{mmu1} holds.
\end{proof}

\begin{remark}
Obviously, the conditions of Lemma \ref{tudr66dr} are satisfied when, for example, the function $\beta$ is locally integrable, while the function $\alpha$ is locally bounded. \end{remark}

\begin{theorem}\label{cyd75i} Let the measure $m$ be given by \eqref{yurt} and assume that the conditions of Lemma \ref{tudr66dr} are satisfied. Then, for $\theta\in C_0(X\to\R_+)$,  the corresponding Radon--Nikodym derivative of the
measure $\mu_m$, $\frac{d\mu_m^\theta}{d\mu_m}$, is given by
\begin{equation*}
\frac {d \mu_{m}^\theta}{d\mu_m}(\eta)=\exp\left[\int_Y(1-\theta^{-1}(x))\frac 1{\alpha(x)}d\eta(x)- \int_Y \log(\theta(x))\beta(x)\,d\sigma(x)\right].
\end{equation*}
\end{theorem}

\begin{proof}
We have
\begin{align*}
&\int_Y\log\left(\frac {l(x,\theta^{-1}(x)s_x)}{l(x,s_x)}\right)\,s_x^{-1}d\eta(x)\\
&\quad =\int_Y\left[\log\left(l(x,\theta^{-1}(x)s_x))-\log(l(x,s_x)\right)\right]s_x^{-1}d\eta(x)\\
&\quad=\int_Y\left[\log\left(\beta(x)e^{\frac {-s_x\theta^{-1}(x)}{\alpha(x)}}\right)-\log\left(\beta(x)e^{\frac{-s_x}{\alpha(x)}}\right)\right]s_x^{-1}d\eta(x)\\
&\quad=\int_Y[-s_x\theta^{-1}(x)+s_x]\frac 1 {\alpha(x)}s_x^{-1}d\eta(x)\\
&\quad =\int_Y(1-\theta^{-1}(x))\frac 1{\alpha(x)}d\eta(x),
\end{align*}
 and by \eqref{buytr687}
\begin{align*}
&\int_Y\int_{\mathbb{R}_+}\frac{(l(x,s)-l(x,\theta^{-1}(x)s)}{s}\,ds\,d\sigma(x)\\
&\quad = \int_Y\int_{\mathbb{R}_+}\frac {\beta(x) e^{ -\frac s{\alpha(x)}}-\beta(x) e^{-\frac {s\theta^{-1}(x)}{\alpha(x)}}}s\,ds\,dx\\
&\quad=\int _Y\beta(x)\left[-\int_{\mathbb R_+}\left(\frac {e^{\frac s{\alpha(x)}(1-\theta^{-1}(x))}-1}s\,\right )e^{-\frac s{\alpha (x)}}ds\right]d\sigma(x)\\
&\quad =\int_Y\beta (x)\log\left(1-\alpha(x)\frac {1-\theta^{-1}(x)}{\alpha(x)}\right)d\sigma(x)\\
&\quad=\int_Y\beta(x)\log(\theta^{-1}(x))\,d\sigma(x)\\
&\quad =-\int_Y\beta(x)\log(\theta(x))\,d\sigma(x).
\end{align*}
Thus, we have
\begin{equation*}
\int_Y\int_{\mathbb{R}_+}\frac{(l(x,s)-l(x,\theta^{-1}(x)s)}{s}\,ds\,d\sigma(x)=-\int_Y\beta(x) \log(\theta(x))d\sigma(x).
\end{equation*}
Now, by Theorem~\ref{hout8r6ode}, the statement follows.
\end{proof}

\begin{remark}\label{vytd7i5}
Note that, for any $A\in\mathcal B_0(X)$ such that $\sigma(A)>0$ and $\beta(x)>0$ for all $x\in A$, we have $m(A\times\R_+)=\infty$.
\end{remark}


\subsubsection{Completely random measures with a L\'evy measure of logarithmic type near zero}\label{lgrthm1}

Let us consider another example of a quasi-invariant measure. 
Let $Y\in\mathcal B(X)$. 
Consider measurable functions $\alpha:Y\to\R_+$ and $\beta:Y\to\R_+$.
Let $\epsilon\in(0,e^{-1})$ and we define, for $(x,s)\in X\times\R_+$,
 \begin{equation}\label{mmmmuyt65}
l(x,s)=\begin{cases}
\beta(x)(-\log s)^{-\alpha(x)},& x\in Y,\ s\in(0,\epsilon), \\
g(x,s),& x\in Y,\ s\in [\epsilon,\infty),\\
0,& x\not\in Y,\ s\in\mathbb R_+,\end{cases}
\end{equation}
so that on $Y\times(0,\epsilon)$
\begin{equation}\label{ytrumh76}
dm(x,s)=\beta(x)\,\frac{(-\log s)^{-\alpha(x)}}s\,d\sigma(x) \,ds,
\end{equation}
and on $Y\times [\epsilon,\infty)$
\begin{equation}\label{myhtunit56}
dm(x,s)=\frac{g(x,s)}s\,d\sigma(x)\,ds.
\end{equation}
Here we assume that the function $g(x,s)$ is strictly positive and satisfies
\begin{equation}\label{uyfr7}
\int_A\int_\epsilon^\infty \frac{g(x,s)}{s}\,ds\, d\sigma(x)<\infty\end{equation}
for all $A\in\mathcal B_0(Y)$.

\begin{lemma}\label{mkuyt1} Let $\beta\in L^1_{\mathrm{loc}}(Y,\sigma)$. Then the  measure $m$  with the function $l(x,s)$ given by \eqref{mmmmuyt65} satisfies \eqref{Q3}.
\end{lemma}

\begin{proof}
By \eqref{uyfr7}, we only need to check that, for any $A\in\mathcal B_0(Y)$
	$$\int _{A\times (0,\epsilon]}l(x,s)\,d\sigma(x)\,ds=
	\int _{A\times (0,\epsilon]}\beta(x)(-\log s)^{-\alpha(x)}\,d\sigma(x)\,ds
	<+\infty.$$
	But, for all $ s\in(0,e^{-1}]$,
	$-\log s\ge 1$,
and since $\alpha(x)>0$, $	( -\log s)^{-\alpha(x)}\le1$.
		Hence, the statement trivially follows.
\end{proof}

\begin{proposition}\label{ute64u38} Let $\beta\in L^1_{\mathrm{loc}}(Y,\sigma)$.
	Then the measure $m$  with the function $l(x,s)$ given by \eqref{mmmmuyt65}  satisfies the conditions of Corollary~\ref{cor2}.
	
\end{proposition}
\begin{proof}
	Let us set $l_1(x,s) =l(x,s)$ and $l_2(x,s)=0$. It suffices to show that, for each $n\in \mathbb N$ and $A\in\mathcal B_0(Y)$,
	$$\int_A\int_0^{\frac \epsilon n}\sup_{u\in[\frac s n,sn]}\left|\frac {\partial}{\partial  u}l(x,u)\right|d\sigma(x)\,ds<\infty.$$
	We have
$$	\frac \partial{\partial u}l(x,u)=\frac {\beta(x)\alpha(x)(-\log u)^{-\alpha(x)-1}}u.$$
Hence, for each $s\in(0,\frac\epsilon n)$,
\begin{align*}
\sup_{u\in[\frac s n,sn]}\left|\frac {\partial}{\partial  u}l(x,u)\right|&=\beta(x)\alpha(x)\sup_{u\in[\frac s n, sn]}\frac {(-\log u)^{-\alpha(x)-1}}u\\
&\le \beta(x)\alpha(x)\left(\sup_{u\in[\frac s n, sn]}\frac1u\right)
\left(\sup_{u\in[\frac s n, sn]}(-\log u)^{-\alpha(x)-1}\right)\\
&=\beta(x)\alpha(x)\frac n{s(-\log (sn))^{\alpha(x)+1}}.
\end{align*}
Then we have
\begin{align*}
&\int_Ad\sigma(x)\int_0^{\frac \epsilon n}ds\,\beta(x)\alpha(x)\frac n s\frac 1 {(-\log (sn))
^{\alpha(x)+1}}\\
&\quad=n\int_Ad\sigma(x)\beta(x)\alpha(x)\int_{-\infty}^{\log \epsilon}ds\,\frac 1{(-s)^{\alpha(x)+1}}\\
&\quad=n\int_A d\sigma(x)\frac {\beta(x)}{(-\log \epsilon)^{\alpha(x)}}\\
&\quad \le \int_A \beta(x)\,d\sigma(x) <\infty.
\end{align*}
Therefore, the conditions of Corollary~\ref{cor2} are satisfied.
\end{proof}

We finish this part with the following observation, which we will use later on.

\begin{proposition}\label{prop2}	(i)
 Assume that $\alpha(x)>1$ for all $x\in X$ and
$$\int_A\frac {\beta(x)} {\alpha(x)-1}\,d\sigma(x)<\infty,$$
for each $A\in \mathcal{B}_0(Y)$. Then
$$m(A\times\mathbb R_+)=\int_Ad\sigma(x)\int_{\mathbb R_+}\frac {l(x,s)}sds<\infty.$$
(ii) Assume that $\alpha(x)\le 1$ for all $x\in X$. Then, for each $A\in \mathcal B_0(Y)$ with $\sigma(A)>0$, we have
$$m(A\times\mathbb R_+)=\int_A d\sigma(x) \int_{\mathbb R_+}\frac {l(x,s)}sds=+\infty.$$
\end{proposition}

\begin{proof} For each $A\in \mathcal{B}_0(Y)$, we have
	$$
	\int_A d\sigma(x)\int_0^{+\infty} ds\,\frac {l(x,s)} s ds=\int_A d\sigma(x) \int_0^\epsilon ds\,\frac {l(x,s)}s+\int_A d\sigma(x)\int_\epsilon^{+\infty}\frac {g(x,s)} s.
	$$
	By \eqref{uyfr7} the second integral on the right hand side is finite. Hence, we need to calculate the first integral on the right hand side.

(i) We have
\begin{align*}
\int_A d\sigma(x)\int_0^\epsilon ds\,\frac {l(x,s)}s&=\int_A d\sigma(x)\int_0^\epsilon\frac{ds} s\frac{\beta(x)}{(-\log s)^{\alpha(x)}}\\
&=\int_A\frac 1 {\alpha(x)-1}(-\log \epsilon)^{-\alpha(x)}d\sigma(x)\\
&\le \int_A\frac {\beta(x)}{\alpha(x)-1}\,d\sigma(x) <+\infty.
\end{align*}

(ii)  We have
\begin{align*}
\int_A d\sigma(x)\int_0^\epsilon ds\,\frac {l(x,s)}s
&=\int_A d\sigma(x)\,\beta(x)\int_{-\log \epsilon}^{+\infty}ds\frac 1{s^{\alpha(x)}} =+\infty.
\end{align*}
\end{proof}

\subsubsection{Completely random measures with a L\'evy measure of power type near zero}\label{vuyfr576re7i}
Let $Y\in \mathcal B(X)$.
Let  functions $ \alpha:Y\to(0,1)$ and $\beta:Y\to \mathbb{R}_+$ be measurable. Let $\epsilon\in(0,1)$. We define for $(x,s)\in X\times \mathbb{R}_+$
 \begin{equation}\label{yfddgyt65}
 l(x,s)=\begin{cases}
\beta(x) s^{1-\alpha(x)},& x\in Y,\ s\in(0,\epsilon), \\
 g(x,s),& x\in Y,\ s\in [\epsilon,\infty),\\
 0,& x\not\in Y.\end{cases}
 \end{equation}
 Thus, on $Y\times (0,\epsilon)$,
 \begin{equation}
 \label{hhhdyoe8ei}
 dm(x,s)=\frac {\beta(x)}{s^{\alpha(x)}}\,d\sigma(x)\,ds,
 \end{equation}
and on $Y\times [\epsilon,\infty)$
\begin{equation}
\label{sgyhdghjdj}
dm(x,s)=\frac  {g(x,s)}s\,d\sigma(x)\,ds.
\end{equation}

\begin{lemma}\label{diyu8} Let the  measure $m$  have the function $l(x,s)$ defined by \eqref{yfddgyt65}. Let $\beta\in L^1_{\mathrm{loc}}(Y,\sigma)$. Then $m$ satisfies \eqref{Q3}.
\end{lemma}

\begin{proof}For each $A\in\mathcal B_0(Y)$,
$$\int_A\int_0^\epsilon \beta(x)s^{1-\alpha(x)}\,ds\,d\sigma(x)\le \int_A \beta(x)\,d\sigma(x)<\infty.$$
By \eqref{uyfr7}, the statement follows.
\end{proof}

\begin{proposition}\label{jkgiyl}
	Let the  measure $m$  have the function $l(x,s)$ given by \eqref{yfddgyt65}. Let $\beta\in L^1_{\mathrm{loc}}(Y,\sigma)$. 
Then the measure $m$  satisfies the conditions of Corollary~\ref{cor2}.
\end{proposition}

\begin{proof}
	We set $l_1(x,s)=l(x,s)$ and $l_2(x,s)=0$. Then 
	$$\frac {\partial}{\partial u}l(x,u)=\frac {\beta(x)(1-\alpha(x))}{u^{\alpha(x)}}.$$
Hence,
$$\sup_{u\in\left[\frac sn,sn\right]}\left|  \frac {\partial}{\partial u}l(x,u)\right|=\frac{\beta(x)(1-\alpha(x))n^{\alpha(x)}}{s^{\alpha(x)}}\le \frac{\beta(x)(1-\alpha(x))n}{s^{\alpha(x)}}.$$	
		Hence,  for each $A\in\mathcal{B}_0(Y)$,
	\begin{align*}
	&\int_ A\int_0^{\frac \epsilon n}\sup_{u\in\left[\frac {s }n,sn\right]}\left|\frac \partial {\partial u}l(x,u)\right|\,d\sigma(x)\,ds\\
	&\quad\le \int_ A\int_0^{\frac \epsilon n}\frac{\beta(x)(1-\alpha(x))n}{s^{\alpha(x)}}\,d\sigma(x)\,ds\\
	&\quad=\int_A \beta(x)n\left(\frac\epsilon n\right)^{-\alpha(x)+1}\,d\sigma(x)\\
	&\quad\le n\int_A \beta(x)\,d\sigma(x)<\infty.
		\end{align*}
	Therefore, the conditions of Corollary~\ref{cor2} are satisfied.
\end{proof}

\begin{proposition}\label{nbuiytr67o5}
Let the conditions of Proposition \ref{jkgiyl} be satisfied. 

(i) Assume additionally that $\frac\beta{1-\alpha}\in L^1_{\mathrm{loc}}(X,\sigma)$. Then, for each $A\in\mathcal B_0(Y)$, we have $m(A\times\R_+)<\infty$.

(ii) Assume that $A\in\mathcal B_0(Y)$ and
$$\int_A \frac{\beta(x)}{1-\alpha(x)}\,d\sigma(x)=\infty.$$
Then $m(A\times\R_+)=\infty$.
\end{proposition}

\begin{proof} Analogously to the proof of Proposition \ref{prop2}, we only need to consider the integral 
$$
\int_A d\sigma(x) \int_0^\epsilon ds\,\frac {l(x,s)}s
=\int_A \frac{\beta(x)}{1-\alpha(x)}\,\epsilon^{1-\alpha(x)}\,d\sigma(x).
$$
Noting that
$$\epsilon\le \epsilon^{1-\alpha(x)}\le 1,$$
we easily conclude the statement.
\end{proof}

\section{Quasi-invariance of completely random measures with respect to transformations of atoms}\label{ty77o897}

From now on, we will assume that $X=\R^d$ and $\sigma$ is the Lebesgue measure $dx$. (More generally, we could assume that $X$ is a smooth Riemannian manifold and $\sigma$ is a volume measure on it.)

In this section, we will consider the transformations of the atoms of completely random measures by the action of the group of diffeomorphisms which are identical outside a compact set.

\subsection{General theory}
 
A {\it diffeormorphism of $X=\mathbb{R}^d$} is a bijective mapping
$\varphi:X\to X$
such that both $\varphi$ and $\varphi^{-1}$ are infinitely differentiable. 
We say that a diffeomorphism $\varphi$ has {\it compact support} if there exists a compact set $\Lambda\subset X$ such that
$\varphi(x)=x$ for all $x\in \Lambda^c$. 
 We denote by
 $\Dif_0(X)$ the set of all  diffeomorphisms of $X$ which have compact support.

It is clear that for any $\varphi,\psi\in\Dif_0(X)$, their composition $\varphi\circ\psi $ again belongs to $\Dif_0(X)$. So we define a group product on $\Dif_0(X)$ as the composition of two diffeomorphisms. The neutral element of this group is the identity mapping $e$. Note that the product in this group is non-commutative.

The group $\Dif_0(X)$ naturally acts on $X$: for each $\varphi\in \Dif_0(X)$, $\varphi(x)$ is the action of $\varphi$ on $x\in X$.
Furthermore, the group $\Dif_0(X)$ naturally acts on $\mathbb{M}(X)$, the space of Radon measures on $X$. For each $\varphi\in\Dif
_0(X)$ and $\eta\in\mathbb{M(X)}$, the action of $\varphi$ on $\eta$ is defined by $\varphi ^*\eta$, the pushforward of $\eta$ under $\varphi$;
$$\varphi^*\eta(\Delta)=\eta(\varphi^{-1}\Delta),\quad\Delta\in\mathcal{B}(X).$$
Clearly $\varphi^*\eta\in\mathbb{M}(X)$.

 Let $\eta\in\mathbb{K}(X)$,
\begin{equation}
\label{byyjyhuo727}
\eta =\sum_is_i\delta_{x_i}.\end{equation}
 Then, for $\varphi\in\Dif_0(X)$
 \begin{equation}
 \label{yheehjdyui1}
 \varphi^*\eta=\sum_is_i\delta_{\varphi(x_i)}.\end{equation}
 In particular, $\varphi^*\eta\in\mathbb{K}(X)$, that is the group $\Dif_0(X)$ acts on $\mathbb{K}(X)$.
 
 Note that each $\varphi\in\Dif_0(X)$ transforms the atoms of a discrete measure, leaving the weights without changes. 

If $\mu$ is a probability measure on $\mathbb{K}(X)$, there is a natural question whether $\mu$ is quasi-invariant with respect to the action of $\Dif_0(X)$. If this is indeed the case, one gets a quasi-regular representation of $\Dif_0(X)$ in $L^2(\mathbb{K}(X),\mu)$.

\begin{theorem}\label{rte6ui4i}
	Let $m$ be a measure on $X\times\mathbb{R}_+$ which satisfies \eqref{tyr75ire}, \eqref{cmpl4.1}. Let $\mu_m$ be the corresponding completely random measure, see Corollary~\ref{mygfrcor45}. For each $\varphi\in\Dif_0(X)$, we extend the action of $\varphi$ to $X\times\mathbb{R}_+$ by setting
\begin{equation}
\label{y37hhy9}
X\times\mathbb{R}_+\ni(x,s)\mapsto (\varphi(x),s)\in X\times \mathbb{R}_+,
\end{equation}
which is a smooth diffeomorphism of $X\times \mathbb{R}_+$.
Let $m_\varphi:=\varphi^*m$ be the pushforward of the measure $m$ under \eqref{y37hhy9}. Then $\mu_m$	is quasi-invariant with respect to the action of $\Dif_0(X)$ if and only if, for each $\varphi\in\Dif_0(X)$,
\begin{itemize}
	\item  $m$ and $m_\varphi$ are equivalent;
	\item $\displaystyle\int _{\hat X}\left(\sqrt{\frac {dm_\varphi}{dm}}-1\right)^2dm<\infty.$
\end{itemize}
\end{theorem}

\begin{proof}
	In view of \eqref{byyjyhuo727}, \eqref{yheehjdyui1} and the construction of the measure $\mu_m$, $\mu_m$ is quasi-invariant with respect to $\Dif_0(X)$ if and only if the Poisson measure $\pi _m$ is quasi-invariant under the following action of $\Dif_0(X)$ onto $\Gamma(\hat{X})$:
	\begin{equation}\label{unsgdo7uu}
	\gamma=\{(x_i,s_i)\}\mapsto\varphi\gamma:=\{(\varphi(x_i),s_i)\},
	\end{equation}
	where $\varphi\in\Dif_0(X)$. Note that, for each $\gamma\in\Gamma(\hat X)$, $\varphi\gamma$ indeed belongs to $\Gamma(\hat X)$. 
	
	Let $\varphi^*\pi_m$ be the pushforward of $\pi_m$ under \eqref{unsgdo7uu}. We claim that
	$$\varphi^*\pi_m=\pi_{\varphi^*m}=\pi_{m_\varphi},$$
	i.e., the Poisson measure on $\Gamma(\hat X)$ with intensity measure $m_\varphi$. Indeed, for each $f\in C_0(\hat X)$, we have
	\begin{align*}
	\int_{\Gamma_{pf}(\hat X)}e^{i\langle f,\gamma\rangle}d(\varphi^*\pi_m)(\gamma)&=\int_{\Gamma_{pf}(\hat X)}e^{i\sum_{(x,s)\in\gamma}f(x,s)}d(\varphi^*\pi_m)(\gamma)\\&=\int_{\Gamma_{pf}(\hat X)}e^{i\sum _{(x,s)\in \gamma}f(\varphi(x),s)}d\pi_m(\gamma)\\
	&=\exp\left[\int_{\hat X}\left(e^{if(\varphi(x),s)}-1\right)dm(x,s)\right]\\
	&=\exp\left[\int_{\hat X}\left(e^{if(x,s)}-1\right)dm_\varphi(x,s)\right]\\
	&=\int_{\Gamma_{pf}(\hat X)}e^{i\langle f,\gamma\rangle}d\pi_{m_\varphi}(\gamma).
	\end{align*}
	Now the statement of the theorem immediately follows from the  Theorem~\ref{vufr76}.
\end{proof}

\begin{corollary}\label{ttgvdcv8yrfhi}
	Let the measure $m$ on $\hat X$ be of the form \eqref{Q1},
	let $l(x,s)>0$ for all $x\in X$ and $s\in\R_+$, and let   \eqref{Q3} be satisfied for all $A\in\mathcal B_0(X)$. Let $\mu _m$ be the corresponding completely random measure. Then $\mu_m$ is quasi-invariant with respect to the action of $\Dif_0(X)$ if and only if, for each $\varphi\in\Dif_0(X)$,
	\begin{multline}\label{udyfhfioyhi08}
	\int_{\hat X}\left(\sqrt{\frac{l(\varphi^{-1}(x),s)}{l(x,s)}\,J_\varphi(x)}-1\right)^2\frac{l(x,s)}{s}\,dx\,ds\\
	=\int_{\hat X}\left(\sqrt{l(\varphi^{-1}(x),s)J_\varphi(x)}-\sqrt{l(x,s)}\right)^2\frac{1}{s}\,dx\,ds<+\infty,
	\end{multline}
	where $J_{\varphi}(x)$ is the modulus of the determinant of the Jacobian matrix  of $\varphi.$
\end{corollary}

\begin{proof}
	By the definition of $m_\varphi$, for each $\Delta\in\mathcal B(\hat X)$,
	\begin{align*}
	m_\varphi(\Delta)&
	=\int_X \chi_\Delta(\varphi(x),s)\,dm(x,s)\\
	&=\int_X \chi_\Delta(\varphi(x),s)\,\frac{l(x,s)}s\, dx\,ds\\
	&=\int_X \chi_\Delta(\varphi(x),s)\,\frac{l(\varphi^{-1}(\varphi(x))),s)}s\, dx\,ds\\
&=\int_X \chi_\Delta(x,s)\,\frac{l(\varphi^{-1}(x),s)}s\, J_\varphi(x)\,dx\,ds\\
&=\int_\Delta \frac{l(\varphi^{-1}(x),s)}{l(x,s)}\,	J_\varphi(x)\,dm(x,s).
	\end{align*}
	
		 Therefore, we have the Radon--Nikodym derivative
	\begin{equation}\label{myhnjfuj1}
	\frac {dm_\varphi}{dm}(x,s)=	
	\frac {l(\varphi^{-1}(x)s)}{l(x,s)}\,J_{\varphi}(x).\end{equation}
	 Therefore, the second condition in Theorem~\ref{rte6ui4i} becomes \eqref{udyfhfioyhi08}.
	 \end{proof}

The following result was shown in \cite{ba1}.

\begin{corollary}
	\label{ujujbhuyj7u7}
	Let $m$ be a measure on $X\times\mathbb{R}_+$ of the form
	$$dm(x,s)=dx\,d\lambda(s),$$
	where $\lambda$ is a measure on $\mathbb{R}_+$. 
	Further assume that
	$$\int_{\R_+}\min\{1,s\}\,d\lambda(s)<\infty.$$
	Then $\mu_m$ is quasi-invariant with respect to the action of $\Dif_0(X)$ if and only if $\lambda(\mathbb{R}_+)<\infty$.
\end{corollary}

\begin{proof} Note that \eqref{tyr75ire} and \eqref{cmpl4.1}
are satisfied. In this case,
	$$\frac{dm_\varphi}{dm}(x,s)=J_{\varphi}(x).$$
	Hence,
	\begin{equation}
	\label{klingdjhufh0}
	\int_{\hat X}\left(\sqrt{\frac {dm_\varphi}{dm}}-1\right)^2dm=\int_X\left(\sqrt{J_{\varphi}(x)}-1\right)^2dx\,\lambda(\mathbb{R}_+).
	\end{equation}
	Since the function $\displaystyle\left(\sqrt{J_{\varphi}(x)}-1\right)^2$ is smooth and has compact support in $X$, we have
	$$\int_X\left(\sqrt{J_\varphi(x)}-1\right)^2<\infty.$$
	Hence \eqref{klingdjhufh0} is finite, if and only if, $\lambda(\mathbb{R}_+)<\infty$.
\end{proof}

The following result generalizes Corollary~\ref{ujujbhuyj7u7}.

\begin{corollary}
	\label{bgujuikulp08mj} Let $m$ be a measure on $X\times\mathbb{R}_+$ which satisfies \eqref{tyr75ire}.
	Assume that, for each $\varphi \in \Dif_0(X)$, the measures $m$ and $m_\varphi$ are equivalent. Further assume that
	\begin{equation}
	\label{gvhdscosfyy}
	m(\Lambda\times \mathbb R_+)<\infty,\quad \Lambda\in\mathcal B_0(X).
	\end{equation}
 Then $\mu_m$  is quasi-invariant with respect to the action of $\Dif_0(X)$ and for each $\varphi\in\Dif_0(X)$, the corresponding Radon--Nikodym density is given by
	 \begin{equation}
	 \label{dvuhfdfuh}
	 \frac {d\mu_m^\varphi}{d\mu_m}(\eta)=\prod_{x\in\tau (\eta)}\frac {dm_\varphi}{dm}(x,s_x).
	 \end{equation}
\end{corollary}

\begin{proof} Note that \eqref{gvhdscosfyy} implies  \eqref{cmpl4.1}.
	According to Theorem~\ref{thrma}, to prove quasi-invariance, it suffices to prove that, for each $\varphi\in\Dif_0(X)$,
	$$ \int_{\hat X}\left|\frac {dm_\varphi}{dm}-1\right|dm<\infty.$$
	Choose $\Lambda\in\mathcal B_0(X)$ such that $\varphi(x)=x$ for all $x\in\Lambda^c$. Then
	$$\frac {dm_\varphi}{dm}(x,s)=1,\text{ for all }(x,s)\in\Lambda^c\times \mathbb R_+.$$
	Hence
	\begin{align*}
	\int_{\hat X}\left|\frac {dm_\varphi}{dm}-1\right|dm&=\int_{\Lambda\times\mathbb R_+}\left|\frac {dm_\varphi}{dm}-1\right|dm\\
	&\le\int_{\Lambda\times\mathbb R_+}\left(\frac {dm_\varphi}{dm}+1\right)dm\\
	&=m_\varphi(\Lambda\times \mathbb{R}_+)+m(\Lambda\times \mathbb{R}_+)\\
	&= 2m(\Lambda\times \mathbb{R}_+)<\infty.
	\end{align*}
	Here $\operatorname{id}$ denotes the identity map.

	Formula \eqref{dvuhfdfuh} will follow from formula \eqref{poi4} (see also Remark~\ref{tfye6e645}) if we show
	$$\int _{\hat X}\left(1-\frac{ dm_\varphi} {dm}\right)dm=0.$$
	Choose again $\Lambda\in\mathcal{B}_0(X)$ such that $\varphi$ is equal to the identity on $\Lambda^c$. Then, for any $(x,s)\in \Lambda^c\times \mathbb{R}_+$, we have
	$$\frac {dm_\varphi} {dm}(x,s)=1.$$
	Hence
	$$
	\int_{\hat X}\left(1-\frac {dm_\varphi}{dm}\right)dm=\int_{\Lambda\times \mathbb R_+}\left(1-\frac {dm_\varphi}{dm}\right)dm)=0.\qedhere
	$$
	\end{proof}


\begin{corollary}
\label{corolarymm}
Assume that the measure $m$ satisfies \eqref{Q1} with $l(x,s)>0$ for all $(x,s)\in X\times\R_+$ and  assume that \eqref{gvhdscosfyy} holds. Then $\mu_m$ is quasi-invariant with respect to the action of $\Dif_0(X)$ and for each $\varphi \in \Dif_0(X)$ we have
\begin{equation}
\label{jodwfdhupy865}
\frac {d\mu _m^\varphi}{d\mu _m}(\eta)=\prod_{x\in\tau(\eta)}\frac {l(\varphi^{-1}(x),s_x)}{l(x,s_x)}J_\varphi(x).
\end{equation}
\end{corollary}

\begin{proof}
	Corollary~\ref{corolarymm} follows from Corollary~\ref{bgujuikulp08mj} and \eqref{myhnjfuj1}.
\end{proof}

\begin{corollary}\label{hjf76}
	Let the assumptions of  Corollary~\ref{ttgvdcv8yrfhi} be satisfied. Assume that there exists an open set $\Lambda\subset X$, $\Lambda\neq\emptyset$, such that, for all $x\in\Lambda$,
\begin{equation}\label{vghyft7}
\int_{\mathbb{R}_+}\frac {l(x,s)}{s}\,ds=\infty.\end{equation}
	Assume that, for each $x\in\Lambda$, the limit $\displaystyle\lim_{s\to 0}l(x,s)=:l(x,0)$ exists,
	$\displaystyle l(x,0)\neq0$, and the function $\Lambda\ni x\mapsto l(x,0)$ is continuous. Then the measure $\mu_m$ is not quasi-invariant with respect to the action of $\Dif_0(X)$.
	
\end{corollary}	
	
\begin{proof} Without loss of generality, we may assume that the set $\Lambda$ is bounded.
	Assume that $\mu_m$ is quasi-invariant with respect to $\Dif_0(X)$. Then by Corollary~\ref{ttgvdcv8yrfhi}, for each diffeomorphism $\varphi$ with support in $\Lambda$, we have
	\begin{equation*}
	\int_\Lambda\int_{\mathbb {R}_+}\left(\sqrt{\frac {l(\varphi^{-1}(x),s)}{l(x,s)}J_{\varphi}(x)}-1\right)^2\,\frac {l(x,s)}s\,ds\,dx<\infty.
	\end{equation*}
	Hence, for a.a. $x\in\Lambda$,
	\begin{equation}
	\label{yhhubojfjjkkk}
	\int_{\mathbb R_+}\left(\sqrt{\frac {l(\varphi^{-1}(x),s)}{l(x,s)}J_{\varphi}(x)}-1\right)^2\,\frac {l(x,s)}s\,ds<\infty.
	\end{equation}
	Note that, for each $x\in \Lambda$,
	\begin{equation}\label{hdncgbjkjdhujkiosj}
	\lim_{s\to 0}\left(\sqrt{\frac {l(\varphi^{-1}(x),s)}{l(x,s)}J_{\varphi}(x)}-1\right)^2=\left(\sqrt{\frac {l(\varphi^{-1}(x),0)}{l(x,0)}J_{\varphi}(x)}-1\right)^2.
	\end{equation}
	By \eqref{vghyft7}, \eqref{yhhubojfjjkkk} and \eqref{hdncgbjkjdhujkiosj}, for a.a. $x\in \Lambda$,
	$$\frac {l(\varphi^{-1}(x),0)}{l(x,0)}J_{\varphi}(x)=1,$$
	or equivalently, for a.a. $x\in \Lambda$,
		\begin{equation}
		\label{dtvhxlxccl}
		l(\varphi^{-1}(x),0)=\frac {l(x,0)}{J_{\varphi}(x)}.
		\end{equation}
		By the continuity of the function $l(\cdot,0)$, we get that equality \eqref{dtvhxlxccl} holds, in fact, for all $x\in\Lambda$ and all diffeomorphisms $\varphi\in\Dif_0(X)$ with support in $\Lambda$. 
		
		But equality \eqref{dtvhxlxccl} is impossible. Just choose any $x,y\in\Lambda$ and any diffeomorphisms $\varphi,\psi\in\Dif_0(X)$ with support in $\Lambda$ such that, for some $x\in\Lambda$,
		$\varphi^{-1}(x)=\psi^{-1}(x)=y$
		and $J_{\varphi}(x)\neq J_\psi(x)$.
		Then
		$$l(y,0)=\frac {l(x,0)}{J_{\varphi}(x)}\neq \frac {l(x,0)}{J_\psi(x)}=l(y,0),$$
		which is a contradiction.
	\end{proof}

\begin{corollary}\label{vgyr75i} Let the assumptions of  Corollary~\ref{ttgvdcv8yrfhi} be satisfied. Assume that there exists an open set $\Lambda\subset X$, $\Lambda\neq\emptyset$, such that, for all $x\in\Lambda$,
	$$\int_{\mathbb{R}_+}\frac {l(x,s)}{s}\,ds=\infty.$$
Assume that there exists  a diffeomorphism $\varphi\in\Dif_0(X)$ such that, for all $x\in\Lambda$, we have
$$\lim_{s\to0}\frac {l(\varphi^{-1}(x),s)}{l(x,s)}J_{\varphi}(x)\ne 1.$$
Then the measure $\mu_m$ is not quasi-invariant with respect to the action of $\Dif_0(X)$.
\end{corollary}

\begin{proof}It immediately follows from the assumptions of the corollary that, for this  diffeomorphism $\varphi\in\Dif_0(X)$, we get
$$\int_{\hat X}\left(\sqrt{\frac{l(\varphi^{-1}(x),s)}{l(x,s)}\,J_{\varphi}(x)}-1\right)^2\frac{l(x,s)}{s}\,dx\,ds=\infty.$$
Hence, the condition of Corollary~\ref{hjf76} is not satisfied and the measure $\mu_m$ is not quasi-invariant with respect to the action of $\Dif_0(X)$.
\end{proof}

\subsection{Examples}

\subsubsection{Completely random gamma measures}\label{gufy576ier}
 Just as in subsec.~\ref{tyd576e}, consider the measure $m$ with
$$l(x,s)=\beta(x)e^{-\frac{s}{\alpha(x)}},$$
where $\alpha,\beta:X\to\R_+$. Assume that the function $\beta$ is continuous and $\alpha\in L^1_{\mathrm{loc}}(X)$. This, in particular implies that $\alpha\beta\in L^1_{\mathrm{loc}}(X)$, hence the condition of Lemma~\ref{lmkuy} is satisfied.

Condition \eqref{vghyft7} is evidently satisfied for each $x\in X$. 
 We also evidently have
$$l(x,0)=\lim_{s\to0} \beta(x)e^{-\frac{s}{\alpha(x)}}=\beta(x).$$
Hence, the conditions of Corollary~\ref{hjf76} are satisfied and the measure $\mu_m$ is not quasi-invariant with respect to the action of $\Dif_0(X)$.

\subsubsection{Completely random measures with a L\'evy measure of logarithmic type near zero}\label{utr7rr67}

We consider two cases.

{\it Case 1.} Let $l(x,s)$ be given by formula \eqref{mmmmuyt65} with $Y=X$ and $\alpha$, $\beta$ being continuous functions. Since $\beta\in L^1_{\mathrm{loc}}(X,dx)$, the condition of Lemma \ref{mkuyt1} is satisfied.

Let us also assume that $\alpha(x)\le 1$. By Proposition~\ref{prop2}, we then get that equation \eqref{vghyft7} holds for all $x\in X$.

Let us assume that the function $\alpha$ is not constant. We get
\begin{equation}\label{vytf7}
\lim_{s\to0}\frac {l(\varphi^{-1}(x),s)}{l(x,s)}J_{\varphi}(x)
=J_{\varphi}(x)\,\frac{\beta(\varphi^{-1}(x))}{\beta(x)}\lim_{s\to0}(-\log s)^{\alpha(x)-\alpha(\varphi^{-1}(x))}.
\end{equation}
Choose an open set $\Lambda\subset X$ and a diffemorphism $\varphi\in\Dif_0(X)$ so that, for all $x\in\Lambda$,
$$\alpha(x)>\alpha(\varphi^{-1}(x)).$$
Hence,
$$\lim_{s\to0}\frac {l(\varphi^{-1}(x),s)}{l(x,s)}J_{\varphi}(x)=+\infty.$$
Hence, the condition of Corollary~\ref{vgyr75i} is satisfied and the measure $\mu_m$ is not quasi-invariant with respect to the action of $\Dif_0(X)$.

If the function $\alpha$ is constant, then evidently formula \eqref{vytf7} becomes
$$\lim_{s\to0}\frac {l(\varphi^{-1}(x),s)}{l(x,s)}J_{\varphi}(x)
=J_{\varphi}(x)\,\frac{\beta(\varphi^{-1}(x))}{\beta(x)}.$$
By Corollary~\ref{vgyr75i}, we will conclude that  the measure $\mu_m$ is not quasi-invariant with respect to the action of $\Dif_0(X)$ if we show that 
there exist $\varphi\in\Dif_0(X)$ and an open non-empty set $\Lambda$ such that, for all $x\in\Lambda$,
\begin{equation}\label{tre643}
J_\varphi(x)-\frac{\beta(x)}{\beta(\varphi^{-1}(x))}\ne0.\end{equation}
Since $J_\varphi$ and $\frac{\beta}{\beta\circ\varphi^{-1}}$ are continuous functions, this will follow from the statement that there exist $\varphi\in\Dif_0(X)$ and $x\in X$ such that \eqref{tre643} holds.
But for this, we can easily construct a diffeomorphism $\varphi\in\Dif_0(X)$ such that $\varphi(x)=x$ but $J_\varphi(x)\ne1$.

{\it Case 2.}  Let $l(x,s)$ be given by formula \eqref{mmmmuyt65} with $Y=X$. Assume that the conditions of Proposition~\ref{prop2}, (i) are satisfied. 
In particular $\alpha(x)>1$ for all $x\in X$. Then, by  Proposition~\ref{prop2}, (i) and Corollary \ref{bgujuikulp08mj}, $\mu_m$ is quasi-invariant with respect to the action of $\Dif_0(X)$.

\subsubsection{Completely random measures with a L\'evy measure of power type near zero}\label{yrtetw34}

Let $l(x,s)$ be as in subsec.~\ref{vuyfr576re7i} with $Y=X$. If $\beta\in L^1_{\mathrm{loc}}(X,dx)$ and $\frac{\beta}{1-\alpha}\in L^1_{\mathrm{loc}}(X,dx)$. Then, by Proposition \ref{nbuiytr67o5} and Corollary \ref{bgujuikulp08mj}, $\mu_m$ is quasi-invariant with respect to the action of $\Dif_0(X)$.

\section{Quasi-invariance and partial quasi-invariance with respect to the semidirect product}\label{ydr637}

In this section, we will study quasi-invariance of $\mu_m$ with respect to the semidirect product of the groups $C_0(X\to\R_+)$ and $\Dif_0(X)$. 

\subsection{Quasi-invariance with respect to the semidirect product}
We recall that an {\it automorphism} $\alpha$ of a group $(G,\cdot)$ is a bijective mapping $\alpha:G\to G$ such that, for any $g_1,g_2\in G$, we have $\alpha(g_1\cdot g_2)=\alpha(g_1)\cdot\alpha(g_2)$.

Following \cite{ba1}, we  define the semidirect product of $\Dif_0(X)$ and $C_0(X\to\mathbb R_+)$. The group $\Dif_0(X)$ acts on $C_0(X\to\mathbb R_+)$ by automorphisms. More precisely, for each 
$\varphi \in\Dif_0(X)$, we may define an automorphism of $C_0(X\to\mathbb R_+$
by
$$C_0(X\to\mathbb R_+)\ni\theta\mapsto \alpha(\varphi)\theta=\theta \circ \varphi^{-1}\in C_0(X\to\mathbb R_+).$$

Let $\mathfrak G$ be the Cartesian product of $\Dif _0(X)$ and $C_0(X\to \mathbb R_+)$:
        $$\mathfrak G=\Dif_0(X)\times C_0(X\to \mathbb R_+).$$
  We define a group multiplication on $\mathfrak G$ as follows:
   for any $g_1=(\varphi_1,\theta _1)$,  $g_2=(\varphi_2,\theta_2)\,\in \mathfrak G,$
   we set
   \begin{equation*}
   g_1g_2=(\varphi_1\circ \varphi_2,\theta _1(\theta _2\circ \varphi_1^{-1})).
   \end{equation*}
   Then $\mathfrak G$ becomes a group. One denotes this group by
   $$\mathfrak G=\Dif_0(X)\leftthreetimes C_0(X\to  \mathbb R_+)$$
   and one calls $\mathfrak G$ the {\it semidirect
   product of $\Dif_0(X)$ and $C_0(X\to \mathbb R_+)$ with respect to $\alpha$}.

   The group $\mathfrak G$ naturally acts on $\mathbb M(X)$, the space of Radon measures on $X$: for any $g=(\varphi,\theta)\in\mathfrak G$ and any $\eta\in \mathbb M(X)$, we define the Radon measure $g\eta $ by
   \begin{equation}\label{uytil8}
   d(g\eta)(x):=\theta(x)d(\varphi^*\eta)(x).
   \end{equation}
   Here $\varphi^*\eta$ is the push-forward of $\eta$ under $\varphi$. Note that when $g=(\varphi, \theta )$ acts on $\eta$, we first act on $\eta$ by $\varphi$, i.e., we take $\varphi^*\eta$, and then we act by $\theta$, i.e., we multiply the measure $\varphi^*\eta$ by $\theta$.
    Note that each $g\in\mathfrak G$ maps $\mathbb K(X)$ into $\mathbb K(X)$.

\begin{proposition}\label{hgifvjlfflv}
    Let $\mu$ be a measure on $\mathbb K(X)$ (or $\mathbb M(X)$). The measure $\mu$ is quasi-invariant with respect to $\mathfrak G$ if and only if $\mu$ is quasi-invariant with respect to the action of both groups $\Dif_0(X)$ and $C_0(X\to \mathbb R_+)$. In the latter case, we have, for each $g=(\varphi,\theta)\in\mathfrak G$,
    \begin{equation}\label{tyr7i54o}
    \frac {d\mu^\theta}{d\mu}(\eta)=\frac {d\mu^\varphi}{d\mu}(\theta^{-1}\eta)\frac{d\mu^\theta}{d\mu}(\eta).
    \end{equation}
\end{proposition}

\begin{proof}
	If $\mu$ is quasi-invariant with respect to $\mathfrak G$, then automatically it is quasi-invariant with respect to the action of $\Dif_0(X)$ and $C_0(X\to \mathbb R_+)$, since $\Dif_0(X)$ and $C_0(X\to \mathbb R_+)$ are subgroups of $\mathfrak G$. So assume that $\mu$ is quasi-invariant with respect to $\Dif_0(X)$ and $C_0(X\to\mathbb R_+)$ and let us prove  that $\mu$ is quasi-invariant with respect to $\mathfrak G$.
	
	Let $F:\mathbb M(X)\to [0,+\infty]$ be a measurable function. Let $g=(\varphi,\theta)\in\mathfrak G$. We have, by \eqref{uytil8},
	\begin{align}\nonumber\label{digg2}
	\int_{\mathbb M(X)} F(\eta)d\mu^g(\eta)&=\int_{\mathbb M(X)}F(g\eta)d\mu(\eta)\\
	&=\int_{\mathbb M(X)}F(\theta (\varphi^*\eta))d\mu(\eta)\notag\\
	&=\int_{\mathbb M(X)}F(\theta \eta)d\mu^\varphi(\eta).
	\end{align}
	Since $\mu$ is quasi-invariant with respect to the action of $\Dif_0(X)$, we continue \eqref{digg2} as follows:
	\begin{align*}
	&=\int _{\mathbb M(X)} F(\theta \eta)\,\frac {d\mu^{\varphi}}{d\mu}(\eta)\,d\mu(\eta)\\
	&=\int_{\mathbb M(X)}F(\eta)\frac {d\mu^ \varphi}{d\mu}(\theta ^{-1}\eta)d\mu^\theta(\eta).
	\end{align*}
	Since $\mu$ is quasi-invariant with respect to $C_0(X\to\mathbb R_+)$, we continue: 
		\begin{equation*}
	=\int_{\mathbb M(X)}F(\eta)\frac {d\mu^\varphi}{d\mu}(\theta ^{-1}\eta)\frac {d\mu^\theta}{d\mu}(\eta)d\mu(\eta).
	\end{equation*}
	
	The functions $\frac {d\mu^\theta}{d\mu}$ and $\frac{d\mu^\varphi}{d\mu}$ are strictly positive on $\mathbb M(X)$ $\mu$-almost everywhere. Let
	\begin{align*}
	&A:=\left\{\eta\in\mathbb M(X)\mid \frac{d\mu^\varphi}{d\mu}(\theta ^{-1}\eta)=0\right\},\\
	&B:=\left\{\eta \in\mathbb M(X) \mid \frac{d\mu^\varphi}{d\mu}(\eta)=0\right\}.
	\end{align*}
	As we already said $\mu(B)=0$. But $A=\theta B$. Hence $\mu(A)=0$ because of quasi-invariance with respect to $C_0(X\to\mathbb R_+)$. Thus 
	$$\frac {d\mu^\varphi}{d\mu}(\theta ^{-1}\eta)\frac {d\mu^\theta}{d\mu}(\eta)> 0\quad {\text{$\mu$-a.e.}}$$
	Hence, the probability measures $\mu^g$ and $\mu$ are equivalent and \eqref{tyr7i54o} holds.
\end{proof}

    \begin{theorem}
    	Let $m$ satisfy \eqref{Q1} with $l(x,s)>0$ for all $(x,s)\in\hat X$, \eqref{Q3}, and \eqref{gvhdscosfyy}. Then $\mu_m$ is quasi-invariant with respect to $g=(\varphi,\theta )\in\mathfrak G$ and the corresponding Radon-Nikodym derivative is given  by
    	\begin{multline*}
    	\frac {d\mu_m^g}{d\mu_m}(\eta)=\left(\prod_{x\in\tau(\eta)}\frac {l(\varphi^{-1}(x),\theta^{-1}(x)s_x)}{l(x,\theta ^{-1}(x)s_x)}J_\varphi(x)\right)
    	\\\times \exp\bigg[\int _X\log\left(\frac{l(x,\theta^{-1}(x)s_x)}{l(x,s_x)}\right)
    	    	s_x^{-1}d\eta(x)
    	\text{}+\int_X\int_{\mathbb{R}_+}\frac {\left(l(x,s)-l(x,\theta^{-1}(x)s)\right)} s \,ds\,dx\bigg].
    	\end{multline*}
    \end{theorem}

    \begin{proof}
    		By Corollary~\ref{corolarymm}, $\mu_m$ is quasi-invariant with respect to $\Dif_0(X)$ and 
    		\begin{equation}
    		\label{fyguhfeijk}
    		\frac {d\mu_m^\varphi}{d\mu_m}(\eta)=\prod_{x\in\tau(\eta)}\frac {l(\varphi^{-1}(x),s_x)}{l(x,s_x)}J_\varphi(x).
    		\end{equation}
    	In Corollary~\ref{cor2}, we set $l_1(x,s)=0$ and $l_2(x,s)=l(x,s)$. Then \eqref{mmu1} and \eqref{mmu2} are satisfied and by  Theorem~\ref{hout8r6ode} and Corollary~\ref{cor2}, $\mu_m$ is quasi-invariant with respect to $C_0(X\to\mathbb R_+)$ and
    	\begin{multline}\label{jffhfhj}
    	\frac{d\mu_m^\theta}{d\mu _m}(\eta)=\exp\bigg[\int _X\log\left(\frac{l(x,\theta^{-1}(x)s_x)}{l(x,s_x)}\right)s_x^{-1}d\eta(x)
    	\\
    	\text{}+\int_X\int_{\mathbb{R}_+}\frac {\left(l(x,s)-l(x,\theta^{-1}(x)s)\right)} s \,ds\,dx\bigg].
    	\end{multline}
    Now the statement of the theorem follows from Proposition~\ref{hgifvjlfflv}, \eqref{fyguhfeijk} and \eqref{jffhfhj}.    	
    \end{proof}

\begin{example}
Let $l(x,s)$ be given by formula \eqref{mmmmuyt65} with $Y=X$. Assume that the conditions of Proposition~\ref{prop2}, (i) are satisfied. 
In particular $\alpha(x)>1$ for all $x\in X$. Then by Proposition \ref{ute64u38}  and subsec.~\ref{utr7rr67}, (ii), the measure $\mu_m$ is quasi-invariant with respect to the action of $\mathfrak G$.
\end{example}

\begin{example}
Let $l$ be as in subsec. \ref{yrtetw34}. Then it follows from Proposition~\ref{jkgiyl} subsec.~\ref{yrtetw34} that he measure $\mu_m$ is quasi-invariant with respect to the action of $\mathfrak G$.
\end{example}

\subsection{Partial quasi-invariance with respect to the semidirect product}The following definition is taken from \cite{ba1}.

Let $(\Omega,\mathcal F, P)$ be a probability space, and let $G$ be a group which acts on $\Omega$. We say that {\it the probability measure $P$ is partially quasi-invariant with respect to transformations $g\in  G$} if there exists a filtration $(\mathcal F_n )_{n=1}^\infty$  such that 
\begin{itemize}
	\item $\mathcal F$ is the minimal $\sigma$-algebra on $\Omega$ which contains all $\mathcal F_n$, $n\in\mathbb N$;
	\item For each $g\in G$ and $n\in\mathbb N$, there exists $k\in\mathbb N$ such that $g$ maps $\mathcal F_n$ into $\mathcal F_k$;
	\item For any $n\in\mathbb N$ and $g\in G$, there exists a measurable function $R^{(n)}_g:\Omega\to[0,+\infty]$ such that, for each $F:\Omega\to[0,\infty]$ which is $\mathcal F_n$-measurable,
	\begin{equation*}
	\int_\Omega F(\omega)dP^g(\omega)=\int_{\Omega}F(\omega)R_g^{(n)}(\omega)dP(\omega).
	\end{equation*}
	Here $P^g$ is the push-forward of $P$ under $g$.
\end{itemize}

\begin{remark}
If $P$ is quasi-invariant with respect to the action of $G$, then it is partially quasi-invariant. In this case, just choose $\mathcal F_n=\mathcal F$ and $R^{(n)}_g=\frac{dP^g}{dP}$.
\end{remark}

\begin{theorem}\label{yrde46}
	Assume that the conditions of Theorem~\ref{hout8r6ode} are satisfied. Assume that there exists $\Lambda\in\mathcal B_0(X)$ such that $m(\Lambda\times\mathbb R_+)=+\infty$. Then the measure $\mu_m$ is partially quasi-invariance with respect to the action of the group $\mathfrak G$.
\end{theorem}

\begin{proof}
	The Borel $\sigma$-algebra $\mathcal B(\Gamma_{pf}(\hat X))$ may be identified as the minimal $\sigma$-algebra on $\Gamma_{pf}(\hat X)$ with respect to which any mapping of the following form is measurable:
	 \begin{equation}
	 \label{trtr7m8}
	 \Gamma_{pf}(\hat X)\ni\gamma\mapsto|\gamma\cap\Lambda|,\,\,\,\Lambda\in\mathcal B_0(\hat X),
	 \end{equation} 
	 see e.g. Section 1.1, in particular Lemma 1.4 in \cite{Kal}. 	 
 	 For each $n\in\mathbb N$, we denote by $\mathcal B_n(\Gamma_{pf}(\hat X))$ the minimal $\sigma$-algebra on $\Gamma_{pf}(\hat X)$ with respect to which each mapping of the form \eqref{trtr7m8} is measurable with $\Lambda\subset [\frac 1 n,\infty)\times X$. Obviously $(\mathcal B_n(\Gamma_{pf}(\hat X)))_{n=1}^\infty$ is a filtration and $\mathcal B(\Gamma_{pf}(\hat X))$ is the minimal $\sigma$-algebra on $\Gamma_{pf}(\hat X)$ which contains all $\mathcal B_n(\Gamma_{pf}(\hat X))$.
	 
	 Recall  \eqref{hyutw3r4}. Let  $\mathcal B_n(\mathbb K(X))$ denote the image of $\mathcal B_n(\Gamma_{pf}(\hat X))$ under the mapping $\mathcal R$.
	 Therefore, $(\mathcal B_n(\mathbb K(X)))_{n=1}^\infty$ is a filtration and $\mathcal B(\mathbb K(X))$ is the minimal $\sigma$-algebra on $\mathbb K(X)$ which contains all $\mathcal B_n(\mathbb K(X))$. 
	 
	 The following lemma follows immediately from the definition of $\mathcal B_n(\mathbb K(X))$.

	 \begin{lemma}\label{pfdr1r4}
	 	A function $F$ is $\mathcal B_n(\mathbb K(X))$-measurable if and only if $F$ is $\mathcal B(\mathbb K(X))$-measurable and for each $\eta=\sum_is_i\delta_
	{ x_i}\in \mathbb K(X)$
	\begin{equation}\label{uityi}	 	F(\eta)=F\left(\sum_{i:\,s_i\geq \frac 1 n}s_i\delta_{x_i} \right). \end{equation}
	 	\end{lemma}

		\begin{lemma}\label{yur685t}
		Let $g=(\varphi,\theta)\in\mathfrak G$. Let  $n\in\mathbb N$ and let $k\in\mathbb N$ be such that 
		\begin{equation}\label{vtyre7i5}
		\frac 1k\le\frac 1n\,\inf _{x\in X}\theta(x).\end{equation}
		 Then $g$ maps $\mathcal B_n(\mathbb K(X))$ into $\mathcal B_k(\mathbb K(X))$.
		\end{lemma}

		\begin{proof}
			Let $F:\mathbb K(X)\to[0,+\infty)$ be a $\mathcal B_n(\mathbb K(X))$-measurable function. Thus, by Lemma~\ref{pfdr1r4}, formula~\eqref{uityi} holds.
			 We note that the inverse element of $g=(\varphi, \theta)$ in the algebra $\mathfrak G$ is $g^{-1}=(\varphi^{-1},\theta^{-1}\circ \varphi)$. Let us the consider the function 
			 \begin{equation}
			 \label{ehukffhi}
			 \mathbb K(X)\ni\eta\mapsto F(g^{-1}\eta)\in[0,+\infty).
			 \end{equation}
			 This function is evidently $\mathcal B(\mathbb K(X))$-measurable. Then, by Lemma~\ref{pfdr1r4} and \eqref{vtyre7i5}, for $\eta=\sum_i s_i\delta_{x_i}\in\mathbb K(X)$,
			 \begin{align*}
			 F(g^{-1}\eta)&=F\left(\sum_i\theta^{-1}(\varphi(\varphi^{-1}(x_i)))s_i \delta_{\varphi^{-1}(x_i)}\right)\\
			 &=F\left(\sum_i\theta^{-1}(x_i)s_i\delta_{\varphi^{-1}(x_i)}\right)\\
			 &=F\left(\sum_{i:\,\theta^{-1}(x_i)s_i\geq\frac 1 n}\theta^{-1}(x_i)s_i\delta_{\varphi^{-1}(x_i)}\right)\\
			 &=F\left(\sum_{i:\,s_i\geq\frac 1 n\,\theta(x_i)}\theta^{-1}(x_i)s_i\delta_{\varphi^{-1}(\,x_i)}\right)\\
			 &=F\left(\sum_{i:\,s_i\geq\frac 1 n\inf_{x\in X}\theta(x)}\theta ^{-1}(x_i)s_i\delta_{\varphi^{-1}(x_i)}\right)\\
			 &=F\left(\sum_{i:\,s_i\geq\frac 1 k}\theta^{-1}(x_i)s_i\delta_{\varphi^{-1}(x_i)}\right)\\
			 &=F\left(g^{-1}\bigg(\sum_{i:\,s_i\geq\frac 1 k}s_i\delta_{x_{i}}\bigg)\right).
			 \end{align*}
			 Hence, by Lemma\ \ref{pfdr1r4}, the function $F(g^{-1}\cdot)$ is  $\mathcal B_k(\mathbb K(X))$-measurable.
			 	 	
Let $A\in\mathcal B_n(\mathbb K(X))$ and let $F=\chi_ A$.  
Thus, $F$ is  a $\mathcal B_n(\mathbb K(X))$-measurable function.
Therefore, $F(g^{-1}\cdot)=\chi_A(g^{-1}\cdot)$ is a $\mathcal B_k(\mathbb K(X))$-measurable function. But 
$$\chi_A(g^{-1}\eta)=\chi_{gA}(\eta),$$
which implies $gA\in\mathcal B_k(\mathbb K(X))$.
\end{proof}

	Next, let $F:\mathbb K(X)\to[0,+\infty]$ be measurable with respect to $\mathcal B_n(\mathbb K(X))$. 
	Let $g=(\varphi,\theta)\in\mathscr G$. Then 
		\begin{align}
\int_{\mathbb K(X)}F(\eta)d\mu^g_m(\eta)	&=\int_{\mathbb K(X)}F(g\eta)d\mu_m(\eta)\notag\\
	&=\int_{\mathbb K(X)}F(\theta(\varphi^*\eta))d\mu_m(\eta)\notag\\
	&=\int_{\mathbb K(X)}F(\theta \eta)d\mu_m^\varphi(\eta).\label{tyreu5643}
	\end{align}

Let $k\in\mathbb N$ be chosen so that
\begin{equation}\label{crte65u}
\frac 1k\le\frac 1n\,\inf _{x\in X}\theta^{-1}(x).\end{equation}
 It follows from the proof of this lemma that the function $\eta\mapsto F(\theta\eta)$ is $\mathcal B_k(\mathbb K(X)) $-measurable.

	 By the construction of the $\sigma$-algebra $\mathcal B_k(\Gamma _{pf}(\hat X))$, this $\sigma$-algebra can be identified with the $\sigma$-algebra $\mathcal B\left(\Gamma_{pf}\left(X\times\big[\frac 1 k,+\infty\big)\right)\right)$. More precisely, each set 
	 $$A\in\mathcal B\left(\Gamma_{pf}\left(X\times\big[\frac 1 k,+\infty\big)\right)\right)$$ is identified with the set 
	 
	 \begin{equation*}
	 \Big\{\gamma \in\Gamma_{pf}(\hat X\mid \gamma \cap\left(X\times\big[\frac 1 k,+\infty\big)\right)\in A\Big\}.\end{equation*}

	 Under this identification, the restriction of the Poisson measure $\pi_k$ on $\Gamma_{pf}(\hat X)$ to the $\sigma$-algebra $\mathcal B_k(\Gamma_{pf}(\hat X))=\mathcal B(\Gamma_{pf}\left(X\times\big[\frac 1 k,+\infty\big)\right)$ coincides with the Poisson measure $\pi_{m(k)}$, where $m(k)$ is the restriction of the measure $m$ to $X\times\big[\frac 1 k,+\infty\big)$.
	 
	 Note that, for each $\gamma\in \Gamma _{pf}\left(X\times\big[\frac 1 k,+\infty\big)\right)$ and each $\Lambda\in\mathcal B_0(X)$, 
	 $$\gamma\cap\left(\Lambda\times\left[\frac 1 k,+\infty\right) \right)$$ is a finite set. Hence, for each $\varphi\in\Dif_0(X)$, the mapping	 
	 \begin{equation}
	 \label{hdieyh8}
\gamma=\sum_i\delta_{(x_i,s_i)}\mapsto \varphi\gamma=\sum_i\delta_{(\varphi(x_i),s_i)}
	 \end{equation}
	 maps $\Gamma_{pf}\left(X\times\big[\frac 1 k,+\infty\big)\right)$ into $\Gamma_{pf}\left(X\times\big[\frac 1 k,+\infty\big)\right)$. We denote by $\mu_{m(k)}^{\varphi}$ the pushforward of the Poisson measure $\pi_{m(k)}$ under the transformation \eqref{hdieyh8}. Thus, we get
	 \begin{align}
	 \int_{\mathbb K(X)}F(\theta \eta)\,d\mu^\varphi_m(\eta)&=\int_{\Gamma_{pf}\left(X\times\big[\frac 1 k,+\infty\big)\right)}F(\theta (\mathcal R\gamma))\,d\pi_{m(k)}^\varphi(\gamma)
	 \notag\\
	 &=\int_{\Gamma_{pf}(\hat X)}F(\theta (\mathcal R\gamma))\,d\pi_{m(k)}^\varphi(\gamma)\notag\\
	 &=\int_{\mathbb K(X)}F(\theta \eta)\,d\mu_{m(k)}^\varphi(\gamma).\label{drdsrd56qw}
	 \end{align}
	 
	  By \eqref{Q3}, for each $\Lambda\in\mathcal B_0(X)$, 
	  
	   $$m\left(\Lambda\times\left[\frac 1 k,+\infty\right)\right)<\infty.$$
	   
Hence, by \eqref{tyreu5643}, \eqref{drdsrd56qw}, Theorem~\ref{hout8r6ode} and Corollary~\ref{corolarymm}, 
\begin{align*}
&\int_{\mathbb K(X)}F(\eta)d\mu^g_m(\eta)=\int_{\mathbb K(X)}F(\theta \eta)\,d\mu_{m(k)}^\varphi(\gamma)\\
&\quad=\int_{\mathbb K(X)}F(\theta\eta)\prod_{x\in\tau(\eta):s_x\geq\frac 1 k}\frac {l(\varphi^{-1}(x),s_x)}{l(x,s_x)}J_\varphi(x)\,d\mu_{m(k)}(\eta)\\
&\quad=\int_{\mathbb K(X)}F(\theta \eta)\prod_{x\in\tau(\eta):s_x\geq\frac 1 k}\frac {l(\varphi^{-1§}(x),s_x)}{l(x,s_x)}J_\varphi(x)\,d\mu_{m}(\eta)\\
&\quad=\int_{\mathbb K(X)}F(\theta \eta)\prod_{x\in\tau(\eta):\theta(x)s_x\geq\frac {\theta(x)} k}\frac {l(\varphi^{-1}(x),\theta^{-1}(x)\theta(x)s_x)}{l(x,\theta^{-1}(x)\theta(x)s_x)}J_\varphi(x)\,d\mu_{m}(\eta)\\
&\quad=\int_{\mathbb K(X)}F( \eta)\left(\prod_{x\in\tau(\eta):s_x\geq\frac {\theta(x)} k}\frac {l(\varphi^{-1}(x),\theta^{-1}(x)s_x)}{l(x,\theta^{-1}(x)s_x)}\right)\\
&\qquad\times
\exp\bigg[\int _X\log\left(\frac{l(x,\theta^{-1}(x)s_x)}{l(x,s_x)}\right)s_x^{-1}d\eta(x)
\\
&\quad\qquad+\int_X\int_{\mathbb{R}_+}\frac {\left(l(x,s)-l(x,\theta^{-1}(x)s)\right)} s \,ds\,dx\bigg] d\mu_{m}(\eta)\\
&\quad= \int_{\mathbb K(X)}F( \eta)R_g^{(n)}(\eta)\, d\mu_{m}(\eta),
\end{align*}
	  	where
		\begin{align*}
		R_g^{(n)}(\eta)&=\left(\prod_{x\in\tau(\eta):s_x\geq\frac {\theta(x)} k}\frac {l(\varphi^{-1}(x),\theta^{-1}(x)s_x)}{l(x,\theta^{-1}(x)s_x)}\right)\\
&\qquad\times
\exp\bigg[\int _X\log\left(\frac{l(x,\theta^{-1}(x)s_x)}{l(x,s_x)}\right)s_x^{-1}d\eta(x)
\\
&\quad\qquad+\int_X\int_{\mathbb{R}_+}\frac {\left(l(x,s)-l(x,\theta^{-1}(x)s)\right)} s \,ds\,dx\bigg].
\end{align*}
(Recall that $k$ depends on $n$ through \eqref{crte65u}.)
	\end{proof}

\begin{example}
Let $m$ be as in subsec.~\ref{tyd576e} and let 
the functions $\alpha \beta $ and $\beta$ belong to $L^1_{\mathrm{loc}}(X)$. Further assume that $\beta(x)>0$. Then by
Theorem~\ref{cyd75i}, Remark~\ref{vytd7i5}, and Theorem~\ref{yrde46}, the measure $\mu_m$ is partially quasi-invariant with respect to the action of the group $\mathfrak G$. By subsec.~\ref{gufy576ier}, the measure $\mu_m$ is not quasi-invariant with respect to the action of $\Dif_0(X)$, hence it is not quasi-invariant with respect to the action of $\mathfrak G$.
\end{example}

\section*{Acknowledgments}
EL is grateful to Yuri Kondratiev and Anatoly Vershik for numerous discussions on the subjects of the paper. 

\bibliographystyle{plain}

\end{document}